\documentclass[11pt]{amsart}
\pdfoutput=1
\usepackage[margin=1.2in]{geometry}

\usepackage[utf8]{inputenc} \usepackage[T1]{fontenc}
\usepackage{lmodern}
\usepackage{subfigure}
\usepackage{amsmath, amssymb, amsthm, amsbsy, bbm, amsfonts, enumitem, color, url}

\usepackage{thmtools}
\usepackage{thm-restate} 

\usepackage[all]{xy}
\usepackage[english]{babel}
\usepackage{graphicx, color}
\usepackage{epic}
\usepackage{xcolor}
\usepackage[pdfpagelabels]{hyperref}
\usepackage[capitalize, noabbrev]{cleveref}
\usepackage[Option]{overpic}

\theoremstyle{plain}
\newtheorem{thm}{Theorem}[section]
\newtheorem*{thm*}{Theorem}
\newtheorem{prop}[thm]{Proposition}
\newtheorem{lemma}[thm]{Lemma}

\newtheorem{notation}[thm]{Notation}

\theoremstyle{definition}
\newtheorem{definition}[thm]{Definition} 
\newtheorem{example}[thm]{Example}

\theoremstyle{remark}
\newtheorem{remark}[thm]{Remark}

\numberwithin{equation}{thm}

\newcommand{\defect}{\mathrm{def}} 
\newcommand{\define}{\mathrel{\mathop:}=}
\newcommand{\dd}{\mathrm{d}}
\newcommand{\area}{\mathrm{A}}

\newcommand{\itemref}[2]{\ref{#1}~\ref{#2}}

\begin{document}

\title{The triangle groups $(2,4,5)$ and $(2,5,5)$ are not systolic}
\author{Annette Karrer, Petra Schwer, Koen Struyve}
\address{
Annette Karrer, Karlsruhe Institute of Technology, Englerstrasse 2, 76137 Karlsruhe, Germany\newline
Petra Schwer, Otto von Guericke University Magdeburg, Universitätsplatz 2, 39106 Magdeburg, Germany\newline
Koen Struyve, Ghent, Belgium}
\email{annette.karrer@kit.edu,  petra.schwer@ovgu.de, kstruy@gmail.com}

\date{ \today }

\begin{abstract}
In this paper we provide new examples of hyperbolic but nonsystolic groups by showing that the triangle groups $(2,4,5)$ and $(2,5,5)$ are not systolic. Along the way we prove some results about subsets of systolic complexes stable under involutions. 
\end{abstract}

\maketitle

\section{Introduction}

Systolicity of simplicial complexes is a combinatorial notion of nonpositive curvature which is different in nature than the more widely studied notions of CAT(0) or hyperbolic spaces. 
It was first introduced by Januszkiewicz and \'Swi\k{a}tkowski \cite{Jan+Swi06} as well as Haglund \cite{Haglund}. Recall that a group is \emph{systolic} if it admits a geometric action on a systolic complex. 

It was shown by Przytycki and the second author in \cite{PS} that almost all triangle groups are systolic. They construct an explicit systolic complex on which the triangle groups act geometrically by embedding  their  Davis complex  in a larger, systolic complex to which the geometric action extends. This process of systolizing forced them to exclude the triangle groups $(2,4,4), (2,4,5)$ and $(2,5,5)$. In the same paper it was shown that the group $(2,4,4)$ is not systolic and that hence such a construction will never exist for that group. Whether or not the remaining two groups are systolic remained open.

\subsection{Main results and key ideas} 

In the present paper we close this gap and prove that the two hyperbolic triangle groups $(2,4,5)$ and $(2,5,5)$ are also not systolic. In hindsight this shows that the systolization procedure of \cite{PS} was best possible within the class of triangle groups. 

Recall that a \emph{cycle of length $n$} in a flag simplicial complex $X$ is a set of $n$ edges in $X$  which topologically forms a sphere. We say that the complex $X$ is \emph{$k$-large} for some $k\geq 4$ if every cycle {$C$} of length strictly less than $k$ has a \emph{diagonal}, that is an edge in $X$ connecting two nonconsecutive vertices on {$C$}. 
If a complex $X$ is connected, simply connected, and all its vertex links are $6$-large we will say $X$ is \emph{systolic}. Note that every $6$-large, connected and simply connected complex is systolic. 

Our main result is the following. 

\begin{thm}\label{thm:nonsystolic}
	The triangle groups $(2,4,5)$ and $(2,5,5)$ are not systolic.
\end{thm}

This theorem in particular implies that these two groups are nonsystolic while being hyperbolic groups. To our knowledge they are thus the first 2-dimensional examples of groups with these two properties. In higher dimensions examples of nonsystolic hyperbolic groups had been constructed by Januszkiewicz and Swiatkowski in their work on filling invariants \cite{JSfilling}. However, their methods do not apply in our situation as the tools they use only work in dimensions greater than 2.

In order to show that a given group is not systolic one needs to proof that it cannot act geometrically on a systolic complex. We will obtain \cref{thm:nonsystolic} as a consequence of the following fixed point theorem for actions of these groups on systolic complexes. 

\begin{restatable}{thm}{FixedpointThm} 
	\label{thm:fixedPoint}
	Suppose $\Gamma$ is one of the groups $(2,4,5)$ and $(2,5,5)$. For every simplicial action of $\Gamma$ on a systolic complex $X$ there exists a $\Gamma$-invariant simplex in $X$. In particular, every such action on the geometric realization of $X$ has a global fixed point. 
\end{restatable}

We close this section with two final remarks. 

It is worth noting that the two groups in question seem to share a property with finite groups here as those also satisfy an analogous fixed point theorem \cref{thm:fixedpointthm} that can be obtained from results of Chepoi and Osajda \cite{ChepoiOsajda}. 

Nonsystolicity of the triangle group $(2,4,5)$ has previously been addressed by Andrew Wilks in his unpublished manuscript \cite{Wilks}. His methods would probably extend with some additional work to the other case as well. We have included some remarks comparing both manuscripts in appropriate places. Compare in particular \cref{rem:wilks1,rem:wilks2}.

\subsection{Strategy of proof}

In this section we provide further details on the strategy of the proof of \cref{thm:fixedPoint}. 

Let in the following $\Gamma$ be either the triangle group $(2,4,5)$ or $(2,5,5)$, i.e.\ $\Gamma$ has the following presentation  
\[
\Gamma= \big\langle \{r,s,t\} \;\vert\; x^2 \;\forall\; x\in\{r,s,t\}, (rs)^2,  (st)^j , (rt)^5 \big\rangle \;\text{ where } j= 4 \text{ or }5.
\]

We wish to show that every simplicial action of $\Gamma$ on a systolic complex $X$ has a $\Gamma$-invariant simplex. Throughout this paper we will assume that every action of $\Gamma$ on a complex $X$ is simplicial.  

At first we start by examining subsets of $X$ that are stable under one of the standard generators $r,s$ or $t$.  
More precisely, for a given $u\in\{r,s,t\}$ we examine in Section~\ref{sec:sub-complexes} the structure of the subcomplex $X_u$, called invariance set, spanned by those vertices in $X$ that are either fixed by $u$ or mapped to an adjacent vertex. 
We prove in \cref{prop:XuProperties} that the complex $X_u$ is a systolic, isometrically embedded, full subcomplex of $X$ and that its maximal simplices are stabilized by $u$.
From Chepoi and Osajda's fixed point theorem for simplicial actions of finite groups on systolic complexes, which we restate in \cref{thm:fixedpointthm}, one can deduce that the set $X_u\cap X_v$ is nonempty for any pair of generators $u\neq v\in\{r,s,t\}$. 

\begin{figure}[h!]
	\centering
	\includegraphics[width=0.75\textwidth]{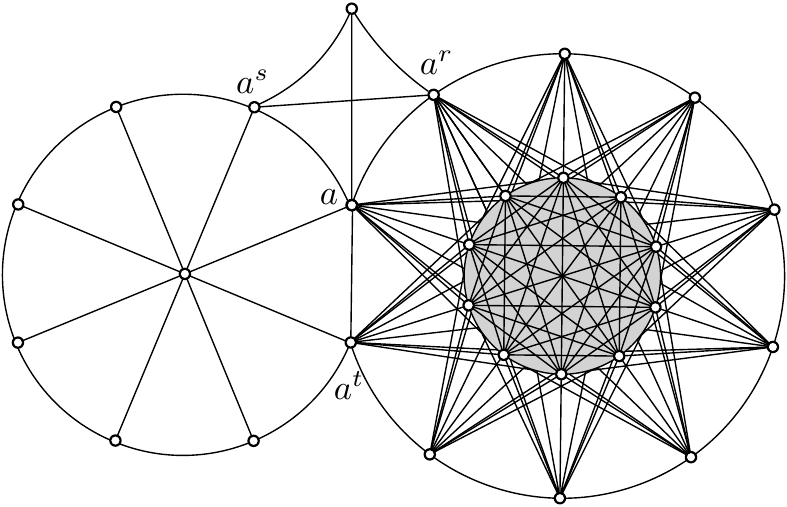}
	\caption{We illustrate \cref{prop:bicycle} in this figure which explains the structure of the orbits of vertices $a$ in pairwise intersections $X_u\cap X_v $ of invariance sets.}
	\label{fig:bicycle}
\end{figure}

Another key ingredient is \cref{prop:bicycle} which shows that the orbit structure for the dihedral subgroups of $(2,4,5)$ or $(2,5,5)$ under a geometric action has a very specific bicycle-like shape as illustrated in \cref{fig:bicycle}. There are three cases for a given $a\in X_u\cap X_v$: either all vertices in $a^{\langle u,v\rangle}$ are connected to a common vertex (see the left "wheel" of the bicycle in \cref{fig:bicycle}) or form a bipartite graph with a clique of the same size (see the right hand side "wheel" in \cref{fig:bicycle}) or they form a clique themselves (illustrated by the bicycle's saddle in \cref{fig:bicycle}). This structure of orbits of vertices in intersections of invariance sets breaks down for  dihedral groups of order 12 and larger which is also the reason why the proof of \cref{thm:fixedPoint} only can work for small triangle groups. Compare also \cref{rmk:highern,ex:bicycle}.

As a next step we construct in \cref{sec:minimalSurface} a minimal surface spanned by a triangle whose sides are geodesics in the invariance sets $X_r, X_s, X_t$. From the fact that the pairwise intersections of two of these sets are nonempty we obtain three vertices $a,b$ and $c$ in the complex $X$ each contained in an intersection of a pair of the invariance sets. 
Any such three vertices $a$, $b$, $c$ hence span a geodesic triangle in $X$ which is fully contained in the set $\bigcup_{u\in\{r,s,t\}} X_u$ which in turn supports a minimal, systolic surface $S$ by a Lemma of Elsner \cite[Le. 4.2]{Elsner09_Flats} which we have restated here as \cref{lem:systolic-surface}. 

The proof of the \cref{thm:fixedPoint} is carried out in \cref{sec:mainProof} and is done by contradiction. 
So we suppose that $\Gamma$ acts {without stabilizing a simplex} on a systolic complex $X$.  
We then choose vertices $x,y,z$, with  
$x\in X_r\cap X_s$, $y\in X_s\cap X_t$ and $z\in X_t\cap X_r$, and the surface  $S$ minimally with respect to length of the bounding geodesics as well as surface area. \cref{prop:existence} then implies that one can choose $x,y$ and $z$ to be pairwise distinct.

Tedious examination of the surface's properties (in \cref{sec:minimalSurface,sec:notasimplex}) and detailed study of the defects at corners in \cref{sec:def-corner} and along the sides of $S$ in \cref{sec:def-sides-sum} will lead us to narrow down the possible structure of the surface to a small list of cases presented in \cref{fig:cases-mainProof}. 
Working through those cases one by one we arrive at a contradiction in each of them which implies that the action cannot be fixed point free. Hence \cref{thm:fixedPoint} {follows}. 

\subsection*{Organization of the paper}

Section~\ref{sec:intro} contains basic definitions and known properties of systolic complexes, which we use throughout the paper. 
For a larger class of groups generated by involutions we study the invariance sets $X_u$ in Section~\ref{sec:sub-complexes} and prove properties which might be of independent interest. 
The specific construction of the minimal surface $S$ is carried out in \cref{sec:minimalSurface}.  By \cref{prop:existence} S is not degenerate if the considered action does not fix a simplex. The fact that $S$ does not just consist of a single $2$-simplex is then shown in \cref{sec:notasimplex} that $S$ is not a single $2$-simplex. The defects of the corners and those along the sides of $S$ are then  studied in \cref{sec:def-corner,sec:def-sides-sum}. 
The proof of the fixed point theorem is finally carried out in Section~\ref{sec:mainProof}. 
Note, that we deal with the two groups simultaneously in most places and that we will say so explicitly, when we don't. Moreover all actions are assumed to be simplicial. 

\subsection*{Acknowledgment}

We wish to thank to Piotr Przytycki for helpful conversations and his comments on an earlier version of the manuscript. While preparing the first version of this paper we learned that Adam Wilks \cite{Wilks} independently showed that the reflection group $(2,4,5)$ is not systolic. The authors acknowledge (partial) funding by the Deutsche Forschungsgemeinschaft (DFG, German Research Foundation) – 281869850 (RTG 2229).

\section{Preliminaries}\label{sec:intro}

The main purpose of this section is to fix notation and to summarize properties of systolic complexes and minimal surfaces therein. 

\subsection{Systolicity}
In this subsection we quickly recall some basic definitions and properties of systolic complexes. 

Let $X$ be a simplicial complex. We call the elements of the $0$-skeleton $X^{(0)}$ of $X$ \emph{vertices} and its $1$-simplices \emph{edges}. So an edge $(a,b)$ is an unordered pair of vertices $a$ and $b$. 
A \emph{path of length $n$} in $X$ is a sequence $(v_0,v_1,v_2\dots v_n)$ of vertices of $X$ where $(v_{i}, v_{i+1})$ is an edge of $X$ for all $i \in \{0,1,\dots,n-1\}$. A \emph{closed path} in $X$ is a path where the first and the last vertex coincide. 
A closed path of length at least three with $v_i\neq v_j$ for all $i,j\in\{1, \ldots, n\}$  is called a \emph{cycle}.  
The \emph{distance} $d(x,y)$ of two vertices is the length of a shortest path from $x$ to $y$. We write $a\sim b$ for vertices $a$ and $b$ that are connected by an edge and will say that $a$ and $b$ are \emph{adjacent}. 

We say that a complex $X$ is \emph{$k$-large} for some $k\geq 4$ if every cycle {$C$} of length strictly less than $k$ has a \emph{diagonal}, that is an edge in $X$ connecting two nonconsecutive vertices on {$C$}. A complex $X$ is \emph{systolic} if it is connected, simply connected, and all its vertex links are $6$-large. 
Note that every $6$-large, connected and simply connected complex is systolic.

We will make repeated use of the combinatorial Gauss-Bonnet Theorem and hence recall its statement from \cite[Le 3.2]{Elsner09_Flats}. 

\begin{prop}[Combinatorial Gauss-Bonnet] \label{prop:GaussBonnet}
	Suppose $\Delta$ is a simplicial disc, then 
	\begin{align*}
	\sum_{v\in \Delta^{(0)}}{\text{def}(v)}=6.
	\end{align*}
	If in addition $\Delta$ is systolic, then the sum of the defects of its boundary vertices is at least $6$, with equality if and only if 
	$\Delta$ has no inner vertices with negative defects. 
\end{prop}

As already noticed by Wilks \cite[Thm.~2.7]{Wilks} the following is a consequence of Theorem C in work of Chepoi and Osajda \cite{ChepoiOsajda}.  

\begin{thm}[Little fixed point theorem]\label{thm:fixedpointthm} 
	For every simplicial action of a finite group $G$ on a systolic complex $X$ there exists a simplex $\sigma$ in $X$ which is invariant under $G$.  
\end{thm}

\subsection{Minimal surfaces}
In this section we define minimal surfaces and collect some of their properties. 

\begin{definition}[Surfaces]\label{def:surface}\label{Def_surface}
	A subcomplex $S \subseteq X$ which is isomorphic (as a simplicial complex) to a triangulated $2$-disc is called a \emph{surface}. The \emph{boundary} of the surface is the cycle $C$ corresponding to the boundary of the $2$-disc. 
	We say that $S$ is \emph{spanned by $C$}.
	The \emph{area} $\area=\area(S)$ of a surface (or 2-disc) $S$ is the number of triangles in $S$. 
	We say that $S$ is $\emph{minimal}$ if there is no other surface spanned by $C$ that has smaller area.
\end{definition}

We compare our notion of a surface with Elsner's definition in \cite{Elsner09_Flats} in the following remark. 

\begin{remark}[Comparison of definitions]
	Elsner defines in \cite{Elsner09_Flats} a \textit{surface spanning a cycle $\gamma$} as a simplicial map $S$ from a triangulated $2$-disc $\Delta$ to $X$ such that $S$ maps $\partial{\Delta}$ isomorphically onto $\gamma$. Furthermore Elsner calls a surface $S: \Delta \to X$ in a systolic complex X \textit{minimal} if $\Delta$ has minimal area among all surfaces extending $S|\partial{\Delta}$. Translated in the language of Elsner, Januszkiewicz and  \'Swi\k{a}tkowski show in \cite[Le 1.6]{Jan+Swi06} that for every cycle in a systolic complex $X$ there exists a surface which is injective on each simplex of the triangulation of $\Delta$. Elsner also proves existence of minimal surfaces and shows that their pre-images are systolic disks, compare \cite[Le. 4.2]{Elsner09_Flats}. It is not hard to see that minimality of the map combined with injectivity on the simplices imply that a minimal surface (in the sense of Elsner) is an injective map. Hence it makes sense to define surfaces as subcomplexes of a complex itself.
\end{remark}

In the following we always work with \cref{def:surface}. 
The following is a reformulation of \cite[Le. 4.2]{Elsner09_Flats} to our statement and holds with almost the same proof.

\begin{lemma}[Systolicity of minimal surfaces]\label{lem:systolic-surface}
	Every cycle $C$ in a systolic complex spans a minimal surface which will necessarily be systolic. 
\end{lemma} 

\subsection{Defect} A big technical piece of work in the proof of \cref{thm:fixedPoint} is the study of defects of vertices and sums of defects of vertices along bounding geodesics of the constructed minimal surface. It is defined as follows. 

\begin{definition}[Defect of vertices in a disc]\label{def:def-vertex}
	Let $\Delta$ be a simplicial $2$-disc. For any vertex $v \in \Delta$ the \emph{defect of $v$}  is defined by the following formula: 
	\begin{align*}
	\defect(v)=
	\begin{cases}
	6-|\{\text{triangles in }\Delta\text{ containing } v \}| & \text{if } v \notin \partial{\Delta}, \\
	3-|\{\text{triangles in }\Delta\text{ containing } v \}| & \text{if } v \in \partial{\Delta}.\\
	\end{cases}
	\end{align*} 
\end{definition}

\noindent 
One may think of the defect as a local way to measure how far a complex is from being systolic. Note that each inner vertex of a systolic disc $\Delta$ has nonpositive defect.  

\begin{definition}[Defect along a geodesic]\label{def:def-geodesic}
	Let $\Delta$ be a simplicial $2$-disc and $\gamma$ a path in the boundary of $\Delta$. 
	The \emph{defect along $\gamma$}, denoted by $\defect(\gamma)$, is defined to be the sum of the defects of all of its inner vertices, i.e. all vertices on $\gamma$ different from its endpoints.  If a path has no inner vertices its defect is defined to be  $0$. 
\end{definition}

The following lemma is an immediate consequence of \cite[Fact~3.1]{Elsner09_Flats} and its proof.

\begin{lemma}[Defects along geodesics in the boundary] \label{lem_sumdef}
	Let $\Delta$ be a systolic disc and $\gamma$  a geodesic in $\Delta$ which is contained in $\partial \Delta$. Then:
	\begin{enumerate}
		\item $\defect(v)\leq 1$ for any inner vertex $v$ of $\gamma$.
		\item for all inner vertices $v_i$ and $v_k$ of $\gamma$ with $\defect(v_i)=\defect(v_k)=1$ there exists $i<j<k$ such that $\defect(v_j)<0$.   
		\item $\defect(\gamma)\leq 1$.  
	\end{enumerate}
\end{lemma}

Item two in the lemma says that if there are two inner vertices of defect one on a geodesic $\gamma$, then they are separated by an inner vertex of negative defect.


\section{Invariance sets: Subcomplexes stable under involutions}\label{sec:sub-complexes}

In this section we investigate the behavior of certain subcomplexes that are almost fixed by an involution. The main result in this section which plays a crucial role in the proof of \cref{thm:fixedPoint} is the bicycle property stated in \cref{prop:H}. 

\subsection{Invariance sets}
We define here for arbitrary simplicial involutions their invariance sets and discuss some of their general properties. 
We emphasize that all results in the present section 
hold true for arbitrary simplicial involutions of a systolic complex $X$. 

Here and in the following we denote the image of a simplex $\Sigma\in X$ under a simplicial automorphism $u$ by  $\Sigma^u$ and the image under the product $u \cdot v\in\Gamma$ by $\Sigma ^{vu}$. Note that simplices which are (setwise) fixed under $u$ are contained in $X_u$. 

\begin{definition}[Invariance sets]
	For a  simplicial involution  $u$ on a systolic complex $X$ we define its \emph{invariance set} $X_u$ to be the flag simplicial complex in $X$ spanned by those vertices $x$ in $X^{(0)}$ for which either $x^u=x$ or $x^u \sim x$. 
\end{definition}

\begin{lemma}[u-invariant simplices]\label{lem:u-stableSplx}\label{lemma:2}
	Let $u$ be a simplicial involution of a systolic complex $X$. Suppose $a_1, \ldots, a_k$ is a clique in $X_u$. Then $a_i\sim a_j^u$ for all $1\leq i,j\leq k$. In other words, the vertices $a_1, \ldots, a_k, a_1^u, \ldots, a_k^u$ span a $u$-invariant simplex in $X$.  
	In particular, for any pair of adjacent vertices $a\sim b$ in $X_u$ one has $a^u \sim b$ and $a\sim b^u$. 
\end{lemma}
\begin{proof}
	Recall that either $a_i^u=a_i$ or $a_i^u\sim a_i$ by definition of $X_u$.
	Suppose first that $k=2$.  We want to conclude that then $a_1\sim a_2^u$. This is clear if one of the vertices $a_1, a_2$ is fixed by $u$, as the action of $u$ on $X$ is simplicial.  
	We assume for a contradiction that $a_i \neq a_i^u$ for both $i=1$ and $2$. Then $(a_1,a_1^u,a_2^u,a_2)$ is a 4-cycle which has a diagonal as $X$ is systolic. Since $u$ acts simplicially on $X$, the existence of one of the diagonals implies the existence of the other. Therefore both diagonals are contained in $X$ and the vertices $a_1, a_2, a_1^u, a_2^u$  span a $u$-invariant simplex in $X$. The rest of the statement follows by induction on $k$ and the fact that we can apply the first induction step to any pair $a_i, a_j$. 
\end{proof}

The next lemma shows that commuting involutions give rise to a simplex that is stable under their span. 

\begin{lemma}[$\langle u,v\rangle$-invariant simplices]\label{lem:commuting uv}
	Suppose $u,v$ are commuting involutions on $X$. Then for any $x\in X_u\cap X_v$ the set $x^{\langle u,v\rangle}$ spans a simplex that is invariant under $u$ and $v$. 
\end{lemma}
\begin{proof}
	It is clear that if they span a simplex it must be stable under both $u$ and $v$. 
	As $x\in X_u\cap X_v$ we have that $x\sim x^u$ and $x\sim x^v$. As the action is simplicial we also have that $x^{uv}\sim x^v$ and that $x^{vu}\sim x^u$. But then from the fact that $u$ and $v$ commute we obtain $x^{uv}=x^{vu}$ and the vertices in the orbit $x^{\langle u,v\rangle}$ either form an edge, a triangle or a $4$-cycle. In case they form a  $4$-cycle there must exist at least one and hence both diagonals.  
\end{proof}

\begin{lemma}[u-invariant mid-simplex]\label{lem:dist_2_simplex}
	Let $u$ be a simplicial involution on $X$. For any $x\in X^{(0)}$ with $d(x,x^u)=2$ the vertices adjacent to both $x$ and $x^u$ span a $u$-invariant simplex.  
\end{lemma}
\begin{proof}
	Let $a$ and $b$ two vertices which are simultaneously adjacent to $x$ and $x^u$. Then $(a,x,b,x^u)$ is a $4$-cycle which has a diagonal by $6$-largeness. As $x \nsim x^u$, the vertices $a$ and $b$ are connected by an edge. Since $X$ is flag, $(a,b,x)$ and $(a,b,x^u)$ span a simplex. With the same argument all the common neighbors of $x$ and $x^u$ span a simplex. It is stabilized by $u$ since $u$ preserves distances. 
\end{proof}

\begin{prop}[Geodesics in $X_u$]\label{prop:1}
	Any two vertices $x$ and $y$ in $X_u$ are connected by a (1-skeleton) geodesic in $X$ which is contained in $X_u$.
\end{prop}
\begin{proof}
	The proof is by induction on the distance $n$ of $x$ and $y$. The statement is clear if $n=0$ or $1$. 
	
	Let $x, y$ in $X_u$ be at distance $n$ in $X$. As $X$ is connected there exists a geodesic $\gamma= (x_0 , x_1, \dots ,x_n)$ from $x=x_0$ to $y=x_n$ in $X$. We want to show that $\gamma$ can be chosen in $X_u$. 
	
	If $x_i = x_i^u$ or $x_i \sim x_i^u$ for some $i \in \{1, \dots, n-1\}$, then $x_i \in X_u$ and we can find via the induction hypothesis a geodesic in $X_u$ connecting $x$ and $y$. Hence we assume that no $x_i$ is contained in $X_u$ for all $i \in \{1,\dots,n-1\}$. Let $S$ be a minimal surface spanned on the two geodesics $\gamma$ and $\gamma^u$. We choose $x$, $y$ and  the geodesic $\gamma$ connecting them in such a way that the area of $S$ is minimal.
	
	By Lemma~\ref{lem:systolic-surface} the surface  $S$ is systolic and hence the sum of the defects at its boundary vertices is $\geq 6$ by \cref{prop:GaussBonnet}. Lemma~\ref{lem_sumdef} then implies that $\defect(\gamma)\leq 1$ and $\defect(\gamma^u)\leq 1$. Let $D\define \sum_{v\in\{x, x^u, y, y^u\}} \defect(v)$ be the sum of the defects of $x, x^u, y, y^u$ in $S$, where we omit possible repetition. From what we have argued $D\geq 4$. 
	
	\emph{{Case 1:} $x\neq x^u$ and $\defect(x)+\defect(x^u)\geq 3$.}\\
	\noindent 
	In this case $\defect(x)+\defect(x^u)= 3$ as otherwise $x_1=x_1^u$ which contradicts the assumption that $x_i\notin X_u$ for all $i$. Therefore one of the vertices $x, x^u$ has defect $1$ and the other defect $2$. Hence $x_1 \sim x_1^u$, i.e. $x_1 \in X_u$, which is a contradiction.
	
	\begin{figure}[h!]
		\centering
		\includegraphics{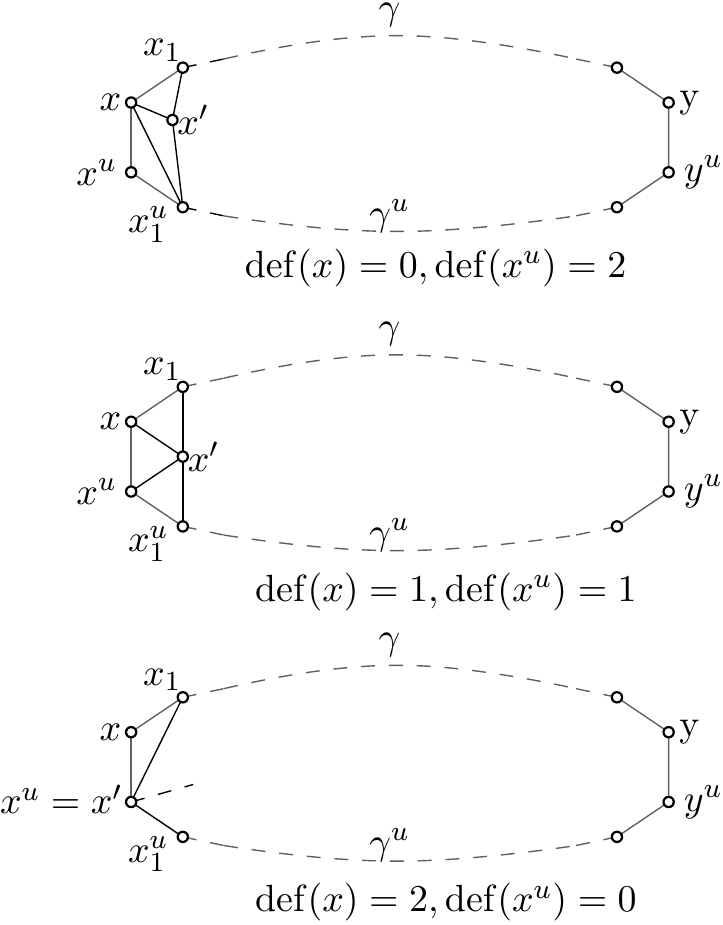}
		\caption{This illustrates Case 2 of the proof of Proposition~\ref{prop:1}.}
		\label{fig:ConnSets_1}
	\end{figure}

	\emph{Case 2: $x\neq x^u$ and $\defect(x)+\defect(x^u)=2$.}\\
	\noindent
	We will obtain that $\dd(x_1,x_1^u) =2$ and that there exists some $x'$ in $S$ adjacent to $x$, $x_1$ and $x_1^u$. This is clearly fulfilled if $x$ has defect $0$ and $x^u$ has defect $2$ or in case that $x$ and $x^u$ have both defect $1$.  In the remaining case, vertex $x$ has defect $2$ and $x^u$ has defect $0$. Then $x^u$ is adjacent to $x$, $x_1$ and $x_1^u$. We illustrate these situations in \cref{fig:ConnSets_1}. 
	
	By Lemma~\ref{lem:dist_2_simplex} the vertices adjacent to both $x_1$ and $x_1^u$ span a $u$-stabilized simplex.
	In particular $x'$ is contained in $X_u$. If $\dd(x',y) < n$, we could find by induction hypothesis a geodesic in $X_u$ connecting $x$ and $y$ via $x'$ which contradicts minimality of $S$. Thus $\dd(x',y) = n$. Then $\gamma' \define (x', x_1, \dots ,x_n \define y)$ is a geodesic connecting $x'$ and $y$ and the minimal surface spanned by $\gamma'$ and $\gamma'^u$ does not contain vertex $x$, i.e. is properly contained in $S$.  Replacing $x$ by $x'$, we obtain a minimal surface with a smaller area than $S$, contradicting the way we have chosen $S$. 
	
	\emph{{Case 3:}} $x=x^u$.\\
	\noindent 
	In this case $x$ has defect at most one, as otherwise $x_1 \sim x_1^u$ or $x_1=x^u_1$ contradicting the fact that $x_1 \notin X_u$. So the defects at $y$ and $y^u$ sum up to at least three and since any vertex on the boundary of $S$ has defect at most two, the vertices $y$ and $y^u$ are distinct. Thus, switching the roles of $x$ and $y$, we are in case $1$.
	
	\emph{{ Case 4:} $x\neq x^u$ and $\defect(x)+\defect(x^u)\leq 2$. }\\
	\noindent 
	This case is covered by cases (1)--(3) by switching the roles of $x$ and $y$. 
	
	Summarizing the observation of the four cases we obtain that all of the $x_i$ must be contained in $X_u$ and the assertion follows. 
\end{proof}

\begin{prop}[Properties of $X_u$]\label{prop:X_s}\label{prop:XuProperties}
	For any $u\in \{r,s,t\}$ the complex $X_u$ is a systolic, isometrically embedded, full subcomplex of $X$ and its maximal simplices are stabilized by $u$.
\end{prop}
\begin{proof}
	The complex $X_u$ is stable under $u$ and by Proposition~\ref{prop:1} its $1$-skeleton is isometrically embedded into $X$.  The fact that $X_u$ is a full subcomplex of $X$ is clear by definition. Now \cite{Elsner09_Isom}[Prop. 3.4] implies that $X_u$ is systolic. The fact that maximal simplices are $u$-stable directly follows from Lemma~\ref{lem:u-stableSplx}. 	
\end{proof}

\subsection{The bicycle property}

In this section we examine the shape of the orbits of vertices in a systolic complex under the action of a small dihedral group. The main result is  \cref{prop:bicycle} which explains the occurring dichotomy and will be referred to as the bicycle property.  The reason why we chose this name is illustrated in \cref{fig:bicycle}.

We first prove a technical lemma. 

Denote by $H$  the dihedral group of order $2n$. Write $u$ and $v$ for the two involutions that generate $H$ and suppose that $H$ acts geometrically on a systolic complex $X$. Let $a$ be a vertex in $X_u \cap X_v$.

%

\begin{lemma}\label{lem:diagonals}
	Denote by $H$ be a dihedral group of order $2n$ with $n \leq 5$ generated by two involutions $u$ and $v$ and suppose $H$ acts on a systolic complex $X$. For any vertex $a$ in $X_u \cap X_v$ the orbit $a^H$ either spans a simplex or the $1$-skeleton of $X$ contains a Hamiltonian cycle of $a^H$, i.e.\ a cycle whose vertex set is $a^H$, of length $2n$ without a diagonal.  
\end{lemma}
\begin{proof}
	Since $a \in X_u \cap X_v$, the complex $X$ contains the closed path \[P=(a,a^u, a^{vu}, a^{uvu},  \dots, a^{(vu)^n}=a).\] This closed path contains all elements of $a^H$. Note that $a$ might be fixed by $u$ or $v$ or both. Hence, two consecutive vertices in $P$ might coincide. Let $C$ be the graph whose vertex set consists of all the vertices that occur in $P$ and whose edge set consists of all edges that occur in $P$. Then $C$ is either a single vertex or a single edge or a Hamiltonian cycle of $a^H$. We have to study the case that $C$ is a Hamiltonian cycle of $a^H$. By Burnside's lemma, $2n$ is divisible by the length of the cycle $C$.
	
	If $\vert a^H \vert < \vert H \vert$ then $\vert a^H \vert \in\{1,2,4, 5\}$ again using Burnside's lemma. 
	By $6$-largeness the cycle $C$ then has at least one diagonal $d$ connecting a pair of vertices at distance $m>1$ on $C$. Using the group action we may conclude that any pair of vertices at distance $m$ on $C$ is connected by an edge. Using $6$-largeness again we conclude that $C$ has all possible diagonals and hence its vertices span a simplex. 
	
	Suppose now that $\vert a^H \vert = \vert H \vert$. If there is no diagonal on $C$ then we are done. Now let $d$ be a diagonal connecting two vertices of distance $1<m\le n$ on $C$. By the same argument as above any pair ov vertices at distance $m$ is connected by an edge. Since $2n< 10$ this implies the existence of cycles of length $4$ or $5$ each of which has one and hence all diagonals by $6$-largeness and using the group action. Thus $C$ spans a simplex. 
\end{proof}

The following example illustrates that \cref{lem:diagonals} is not satisfied by dihedral groups of order larger than $10$.

\begin{example}
	\label{example1}
	Fix a natural number $n$ and construct as follows a simplicial complex which admits a natural action of the dihedral group of order $2n$. Let $C$ be a cycle of length $2n$.  Add to $C$ an edge between every pair of vertices at distance two and let $Z$ be the flag simplicial complex on the cone over the resulting graph. \cref{fig:Example1} shows the complex one obtains for $n=6$. The action of the dihedral group of order $2n$ on $C$ extends to an action on $Z$ having the conepoint as a fixed point. Examining the links of all vertices one can verify that the complex $Z$ is systolic if and only if $n$ is at least $6$. 
\end{example} 

\begin{figure}[h!]
	\centering
	\includegraphics[]{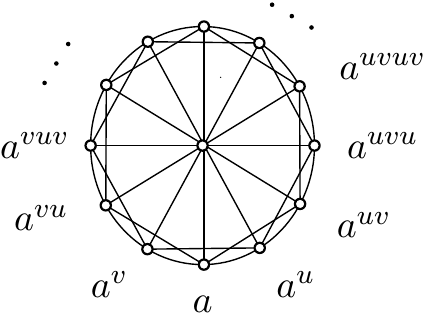}
	\caption{ An example that not all dihedral groups satisfy Lemma~\ref{lem:diagonals}. The dihedral group of order $12$ acts on the pictured systolic complex.}
	\label{fig:Example1}
\end{figure}

We now prove one of the main properties needed for the proof of \cref{thm:fixedPoint}. 

\begin{prop}[bicycle property]\label{prop:H}\label{prop:bicycle}
	Fix $n\leq 5$ and let $H$ be the dihedral group of order $2n$ generated by involutions $u$ and $v$. Suppose further that $H$ acts on a systolic complex $X$ and let $a$ be a vertex in $X_u \cap X_v$. Then one of the following two statements is true:
	\begin{enumerate}
		\item The orbit $a^H$ spans a simplex.
		\item  The orbit $a^H$ forms an $\vert H \vert$-cycle $C$ without diagonals and there exists a vertex $b$ in $X$ such that $b^H$ supports an $H$-stabilized simplex $\Sigma$ and such that the vertices of $C$ and the vertices of $\Sigma$ span a complete bipartite graph. Moreover, each vertex of $\Sigma$ is contained in $X_u\cap X_v$. 
	\end{enumerate}
	If $n \in \{1,2\}$, only the first case occurs. 
\end{prop}

\begin{remark}\label{rmk:highern}
	Proposition~\ref{prop:H} plays a crucial role in the proof of our main result. The key features we need there is on one hand the property that two of the three standard Coxeter generators commute and on the other hand that any pair of generators generates a dihedral group with the properties listed in Proposition~\ref{prop:H}.
	Note that the groups (2,3,3), (2,3,4) and (2,3,5) also show this behavior. But in contrast to (2,4,4) (2,4,5) and (2,5,5) they are finite. 
	We deal with larger $n$ in Example~\ref{ex:bicycle} below. 
\end{remark}

We are now ready to prove \cref{prop:bicycle}. 

\begin{proof}
	Let $C$ be the Hamiltonian cycle on $a^H$ as obtained in \cref{lem:diagonals} and $S$ a minimal surface spanned by $C$. We aim to prove existence of a vertex $b$ with the following three properties: 
	\begin{enumerate}
		\item $b \in X_u \cap X_v$
		\item $b$ is adjacent to three consecutive vertices on $C$
		\item $b^H$ spans a simplex. 
	\end{enumerate}
	
	It then remains to prove that the vertices of $C$ and those of $b^H$ span a complete bipartite graph. We obtain this from the $6$-largeness of $X$ and the way in which the dihedral group acts. Assume without loss of generality that $b$ is adjacent to $a$, $a^u$ and $a^v$. Using the group action we obtain that $a$ is connected to every vertex of $b^H$. But then $X$ contains the closed path $(a,a^v,b^{uv},b^u)$ and $6$-largeness implies that $a~\sim~b^{uv}$. Analogously we obtain that $a \sim b^{vu}$ and also that  $a^v \sim b^{vuv}$. We then conclude that $X$ contains the closed path $(a,a^v,b^{vuv},b^{vu})$. Using  $6$-largeness again we obtain that $a \sim b^{vuv}$ and with similar arguments that $b \sim b^{uvu}$. Repeating these steps we can show that $a$ is adjacent to every element of $b^H$.

	It thus remains to prove that we have a vertex $b$ satisfying the listed properties. 
	This is done by an examination of the defects of vertices on the boundary of $S$. There are two cases. We study two cases.
	
	\emph{Case 1: The boundary of $S$ contains three consecutive vertices of defect $1$}\\
	\noindent 
	Using the group action we may assume that these vertices are $a$, $a^u$ and $a^v$. Let $b$ the neighbor of $a$ in the interior of $S$. Since $a^u$ and $a^v$ have defect $1$ we obtain that $b\sim a^u$, $b \sim a^v$,  $b \sim a^{vu}$ and $b \sim a^{uv}$. Therefore all these $5$ consecutive vertices on $C$ are adjacent to $b$. We show that $b^H$ spans the simplex we are looking for.  
	
	Using the  group action we have that $b^u \sim a^v$ and $a^u \sim b^v$. 
	As $X$ contains the cycles $(b,a^u,b^u,a^v)$ and $(b,a^u,b^v,a^v)$ and as $C$ does not have diagonals, $b$ has to be either equal or adjacent to $b^u$ and $b^v$, i.e. $b \in X_v\cap X_u$.  Furthermore $X$ contains the cycle $(a^u, b^u, a^v, b^v)$ and since $C$ does not have diagonals, $b^u$ is adjacent to $b^v$. Thus the natural Hamiltonian cycle of $b^H$ has a diagonal. We apply Lemma~\ref{lem:diagonals} to conclude that $b^H$ is a simplex.

	\emph{Case 2: The boundary of $S$ does not contain three consecutive vertices of defect $1$}\\
	\noindent
	As a first step we will show that $S$ is isometric to the surface pictured in Figure \ref{fig:Sonderfall}. 
	The vertex $b$ pictured in Figure \ref{fig:Sonderfall} is adjacent to $a$, $a^u$, $a^v$ and $a^{vu}$. We will prove that $b \in X_u \cap X_v$ and that $b^H$ spans a simplex. 
	
	In order to prove that $S$ is isometric to the surface shown in Figure \ref{fig:Sonderfall}observe that $C$ has length $10$: Because $C$ does not contain diagonals, each vertex on the boundary of $S$ has a defect of at most $1$. By Proposition \ref{prop:GaussBonnet} the defects along the boundary of $S$ sum up to at least $6$. If $2n< 10$ and  $\partial S$ does not contain three consecutive vertices of defect $1$, the sum of defects has to be less than $6$ which is a contradiction. Hence the length of $C$ is at least $10$.

	\begin{figure}[ht!]
		\centering
		\includegraphics[width=0.4\textwidth]{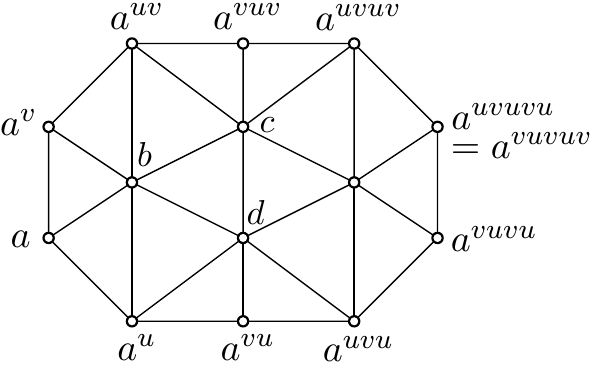} 
		\caption{This surface is a special case appearing in the proof of \cref{prop:bicycle}.}
		\label{fig:Sonderfall}
	\end{figure}

	We show as a next step that $S$ contains exactly four inner vertices. Using the isoperimetrical inequalities of Lemma 3.4 in \cite{Elsner09_Flats} we obtain that any systolic disc $\Delta$ of perimeter $l$  has $\area \le \frac{1}{6}l^2$.
	Hence $S$ contains at most $16$ $2$-simplices. Picks Formula, see \cite[Le 3.3]{Elsner09_Flats}, implies that  $S$ has at most four inner vertices. If $\partial S$ contains at most three inner vertices, the interior of $S$ contains at most one $2$-simplex. This implies that the boundary of $S$ contains at most three vertices of defect $0$ and thus three consecutive vertices of defect $1$ which contradicts our assumptions. Therefore, $S$ contains four inner vertices. Because these $4$ vertices span a subcomplex consisting of two $2$-simplices, exactly four  vertices on the boundary of $S$ have defect $0$ and $6$ vertices on the boundary of $S$ have defect $1$. The surface shown in \cref{fig:Sonderfall} is the only systolic complex satisfying these conditions which implies that $S$ is isometric to this surface. 
	
	The vertex $b$ in \cref{fig:Sonderfall} is adjacent to $a$, $a^u$, $a^v$ and $a^{vu}$. We will prove that $b \in X_u \cap X_v$ and that $b^H$ spans a simplex. 
	
	First we prove that $b \in X_u \cap X_v$. 
	As $X$ contains the closed $4$-path $(b,a,b^v, a^{uv})$ we obtain that $b \sim b^v$ or $b = b^v$ by $6$-largeness. It remains to prove that $b = b^u$ or that $b \sim b^u$. If $b = b^u$, we are done. Hence we assume that $b \neq b^u$. 
	Since $C$ does not contain a diagonal one has that $a\nsim a^{vu}$. If $b \sim a^{vu}$ we can exchange $S$ with a surface whose boundary contains three consecutive vertices of defect $1$. Then $b \sim b^u$ by Case $1$. Hence we suppose that $b \nsim a^{vu}$. Furthermore, $a \nsim d$.  Otherwise $(a,a^v,a^{uv},c,d)$ is a $5$-cycle and by $6$-largeness it can be triangulated to a surface containing three $2$-simplices. These three $2$-simplices then lead to the existence of a surface with boundary $C$ that has less triangles than $S$ which contradicts minimality. Thus,  $a \nsim d$. By $6$-largeness, the  cycle $(a,b,d,a^{vu},b^u)$ contains the remaining two diagonals, i.e. we have $b \sim b^u$ and $b^u \sim d$. It follows that $b \in X_u \cap X_v$.
	
	It remains to prove that $b^H$ spans a simplex.  By Lemma \ref{lem:diagonals}, $b^H$ spans a simplex if the Hamiltonian cycle of $b^H$ contains a diagonal. We prove this by showing that $b^v \sim b^u$. 
	This follows from the existence of a certain $5$-cycle. To form this $5$-cycle, we need that $c^v\sim d$.  To prove this we consider the closed path $(b^v,b,c^v,a^{vu},d)$. If $b \sim a^{vu}$ or $b^v \sim a^{vu}$,  we can exchange $S$ with a surface whose boundary contains three consecutive vertices of defect $1$. Hence we suppose that $b \nsim a^{vu}$ and $b^v \nsim a^{vu}$. Then $c^v\sim d$, because otherwise there is a $4$- or $5$-cycle without a diagonal. 
	Recall that we have seen already that $b^u \sim d$. Hence, $X$ contains the closed path $(a, b^v, c^v, d, b^u)$. We show that this closed path contains the desired diagonal connecting $b^v$ and $b^u$. Recall that we have proven that $a \nsim d$. Analogously we see that $a \nsim v^c$. Hence $b^v \sim b^u$ because otherwise $X$ contains a $4$- or $5$-cycle without a diagonal. 
\end{proof}

\cref{prop:bicycle} is not true for dihedral groups of order larger than $10$ as illustrated in the next example. 

\begin{example}\label{ex:bicycle}
	Consider the flag simplicial complex whose $1$-skeleton is obtained as follows. We take the $1$-skeleton of the simplicial complex of Example \ref{example1}. This $1$-skeleton consists of a vertex $v$ and a cycle $C= (x_1,x_2,\dots,x_{2n-1})$. We add a further  cycle $C'=(y_1,y_2,\dots,y_{2n-1})$ of length $2n$. We connect each vertex $y_i$ on $C'$ with the vertices $x_{m(i-1)}$, $x_{m(i)}$ and $x_{m(i+1)}$, where $m(i)$ denotes $i \text{ mod }{2n}$, $i \in \{1,\dots,2n-1\}$. The example is illustrated in \cref{fig:Example} for $n = 6$ . The obtained simplicial complex is systolic if and only if $n$ is at least $6$. Accordingly, a vertex $a$ on the cycle $C'$ does not satisfy the properties of the last proposition if and only if $n$ is at least $6$.
\end{example}

\begin{figure}[ht!]
	\centering
	\includegraphics{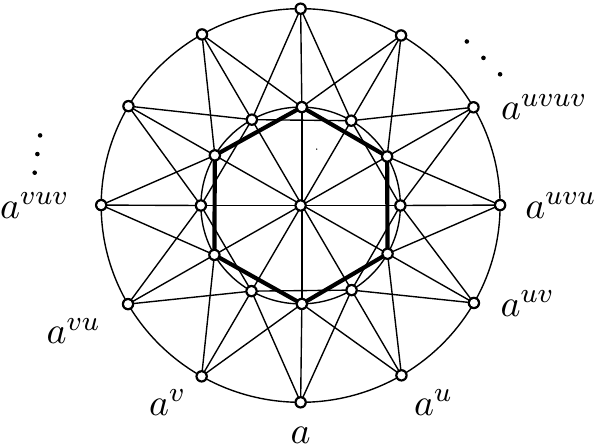}
	\caption{The dihedral group of order $12$ acts on the pictured systolic complex and does hence not satisfy the assertion of \cref{prop:H}. Compare also \cref{ex:bicycle}.}
	\label{fig:Example}
\end{figure}

\begin{remark}\label{rem:wilks1}
	The bicycle property discussed in \cref{prop:bicycle} is related to properties Wilks proves in Lemmas~4.2,~4.6 and~4.8 of~\cite{Wilks}. In addition what we prove in Lemma~\ref{lem:diagonals} is similar to what is done in Lemmas~4.3 and~4.4 in \cite{Wilks}. 
\end{remark}

\section{Construction of a minimal surface $S$}\label{sec:minimalSurface} 

The aim of the present section is to show existence of a very specific minimal surface in $X$. We start by fixing the following notation for the entire section. 

\subsection{The setup}\label{sec:surfaceConstruction}

We begin with fixing some notation. 

\begin{notation}\label{notation:triangle-groups}
	Suppose $\Gamma$ is either $(2,4,5)$ or $(2,5,5)$, that is $\Gamma$ admits one of the two following presentations: 
	\[
	\big\langle r,s,t, \vert r^2 = s^2 = t^2 = (rs)^2 = (st)^j = (rt)^5 \big\rangle \;\text{ where } j= 4 \text{ or }5.
	\]
	Let $X$ be a systolic complex and suppose that $\Gamma$ acts on $X$.
\end{notation}

Using the notation we just introduced  we obtain by Theorem~\ref{thm:fixedpointthm} that the intersection $X_u\cap X_v$ of the respective stabilized subcomplexes of $X$ is nonempty for any choice of $u\neq v$ with $\{u, v\} \subset S$. 

We then choose vertices $x \in X_s \cap X_r$, $y \in X_s \cap X_t$, and $z\in X_r \cap X_t$ and three geodesics connecting them, i.e.\  $\gamma_s: x \rightsquigarrow y$ in $X_s$, $\gamma_t: y \rightsquigarrow z$ in $X_t$, and $\gamma_r: z \rightsquigarrow x$ in $X_r$ in such a way that the resulting cycle is minimal in length.

Note that if $X_r\cap X_s\cap X_t\neq\emptyset$ this means that $x=y=z$, all the geodesics are trivial and the cycle formed by the three geodesics consist of a single vertex. 

In case that $X_r\cap X_s\cap X_t=\emptyset$ the three vertices will be pairwise different and each geodesic is of length at least one.  
Let $C$ be the cycle formed by the concatenation of the three geodesics, i.e.\ $C=\gamma_r\star\gamma_t\star\gamma_s$ and choose a surface $S$ spanned by $C$ that is minimal in area. By \cref{lem:systolic-surface}, such a surface $S$ exists. We will refer to the vertices $x,y,z$ as the \emph{corners} and call the three geodesics  \emph{sides} of $S$. We illustrate this situaiton in Figure~\ref{fig:FlaecheS}.

\begin{figure}[h!]	
	\centering
	\begin{overpic}[scale=1.0,,tics=10]
		{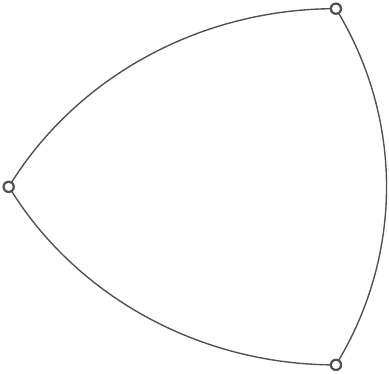}
		\put(-10,48){$x$}
		\put(92,98){$y$}
		\put(92,-2){$z$}
		
		\put(7,85){$\gamma_s \subset X_s$}
		\put(105,48){$\gamma_t\subset X_t$}
		\put(10,5){$\gamma_r\subset X_r$}
	\end{overpic}
	\caption{The surface $S$ with corners $x$, $y$ and $z$ and sides $\gamma_s$, $\gamma_t$ and $\gamma_r$.}
	\label{fig:FlaecheS}
\end{figure}
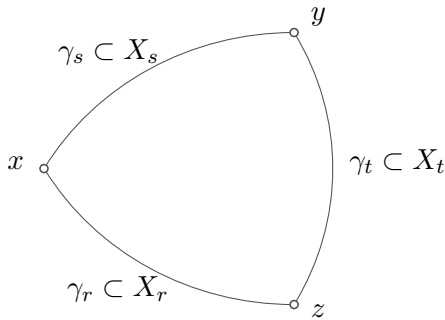

In the next lemma we prove that in case the intersection of the invariance sets is nonempty there exits a global fixed point.  

\begin{lemma}[Existence of a fixed point]
	\label{lemma:fixpoint}
	If $\Gamma$ acts on a systolic complex $X$ with $X_r \cap X_s \cap X_t\neq\emptyset$ then {there exists a $\Gamma$-invariant simplex in $X$. In particular, every such action on the geometric realization of $X$ has a global fixed point.} 
\end{lemma}
\begin{proof}
	First observe that we may assume that $x^{\langle u,w \rangle}$ spans a simplex for all $u,w\in  S$. If $x^{\langle u,w \rangle}$ does not span a simplex for some $u, w\in\{r,s,t\}$ there exists another vertex with the desired properties by Proposition~\ref{prop:H} and $6$-largeness. 
	
	Let $M=x^{\langle r,s \rangle} \cup x^{\langle r,t \rangle} \cup x^{\langle s,t \rangle}$ and let $X'$ be the simplicial subcomplex of $X$ spanned by $M\cup M^t\cup M^s\cup M^r$. We aim to show that $X'$ is a simplex stabilized by $\Gamma$. Figure \ref{fig:fixpoint} serves as an illustration of the situation. 
	
	Proposition~\ref{prop:H} implies that the orbit  $x^{\langle u,w \rangle}$ contains at most $5$ elements for arbitrary $u,w\in\{r,s,t\}$. Using the  $\Gamma$-action we conclude that then either $x^{uw}= x^{wu}$ or  $x^{uw}=x^{uwu}$. Hence the set of vertices in $x^{\langle s,t \rangle} \cup x^{\langle r,t \rangle}$  is contained in $X_t$ and spans a simplex $\tau$ by Lemma~\ref{lemma:2}. Using $6$-largeness one can see that the union $\tau^s\cup \tau^r\cup\tau^{sr}$ of these simplices also forms a simplex and therefore $X'$ is a simplex.  
	
	It remains to prove that $X'$ is stable under $\Gamma$. 
	Fix a vertex $a \in X'$. We need to show that $a^u \in X'$ for all $u \in\{r,s,t\}$. In case $a\in M$ this is clear by construction. 
	In case $a$ is contained in $M^w$, $w\in\{r,s,t\}$ we proceed as follows. If $u = w$, the claim follows directly. Otherwise $a = vuw.x=x^{wuv}$ or $a = vwu.x=x^{uwv}$. So $u$ appears as the second or third letter in the element that is acting.  By Proposition~\ref{prop:H} the orbit  $x^{\langle u,w \rangle}$ contains at most $5$ elements for any choice of $u,w\in\{r,s,t\}$. 
	
	We had observed earlier that either $x^{uw}= x^{wu}$ or $x^{uw}=x^{uwu}$. By Proposition~\ref{prop:H}, the orbit $ux^{\langle v,w \rangle}$ spans a simplex with at most $5$ vertices. The same is true for $wx^{\langle u,v \rangle}$. Hence $(ux)^{vw}= (ux)^{wv}$ or $(ux)^{vw}=(ux)^{vwv}$ and $(wx)^{vu}= (wx)^{uv}$ or $(wx)^{vu}=ux^{vuv}$ which implies that $a^u$ is contained in $X'$. 
	Therefore $\Gamma X'=X'$ and we are done. 
\end{proof}

\begin{figure}[htb]
	\centering
	\includegraphics[width=0.7\textwidth]{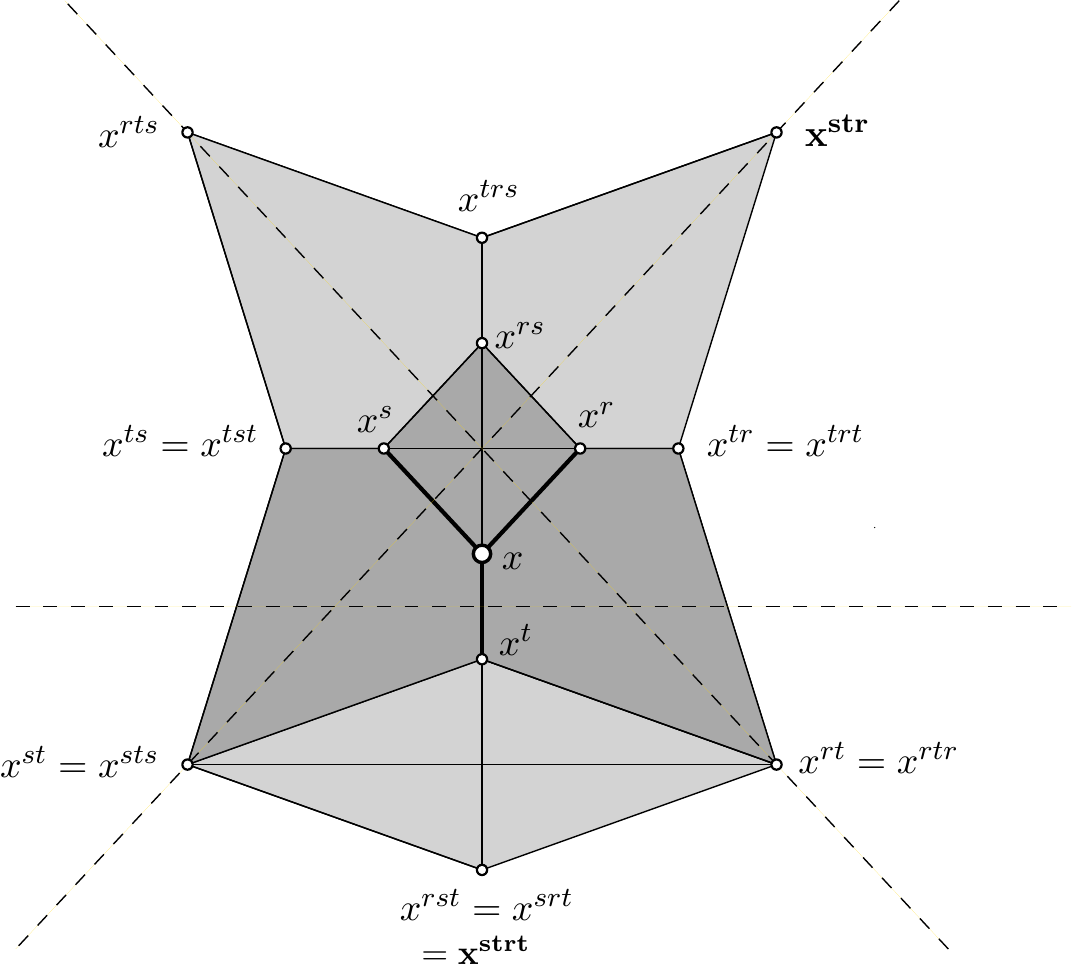}
	\caption{$X'$ as described in the proof of \cref{lemma:fixpoint}. The complex spanned by $M$ is shown in dark gray. }
	\label{fig:fixpoint}
\end{figure} 

\begin{remark}\label{rem:wilks2}
	\cref{lemma:fixpoint} is closely related to \cite[Thm~1.2]{Wilks} which  implies that for the group $(2,4,5)$ the intersection $X_r \cap X_s \cap X_t$ is always nonempty. Note that the proof of Theorem 1.2 in \cite{Wilks} is quite involved. 	
\end{remark}

We will now show that in case the given action {does not fix a simplex} the minimal surface we have constructed will not be degenerate. 

\begin{prop}[Existence of a nondegenerate minimal surface]\label{prop:existence}
	Suppose {$\Gamma$ acts on a systolic complex $X$ without fixing a simplex}. Then $S$ contains at least one $2$-simplex. 
\end{prop}
\begin{proof} 
	Suppose for a contradiction that $S$ does not contain a $2$-simplex and thus is either a single vertex or an edge. In this case, at least two of the three corners of $S$ agree and we have that $X_s\cap X_t\cap X_r\neq\emptyset $. But then \cref{lemma:fixpoint} implies that the action {stabilizes a simplex}  which is a contradiction.
	Hence $S$ contains at least one $2$-simplex.  
\end{proof}

\subsection{The nondegenerate case}

Under the assumption that $\Gamma$ acts {without stabilizing a simplex} on a systolic complex $X$ - and thus that $S$ is nondegenerate - we prove two lemmata. While \cref{lemma:3} will only be used once in the proof of Lemma~\ref{lemma:commute1} the second Lemma~\ref{lemma:4} will be crucial in numerous proofs throughout the paper to show nonexistence of certain diagonals.

\begin{lemma}[Adjacency]\label{lemma:3}
	Suppose $\Gamma$ acts {without stabilizing a simplex} on a systolic complex $X$. Let $S$ be the surface constructed in \cref{sec:surfaceConstruction}. Fix $u \in \{r,s,t\}$ and suppose that $a$ is an inner vertex of the side $\gamma_u$ of defect 1. Then the unique neighbor of $a$ in the interior of $S$ is adjacent to $a^u$.
\end{lemma}
\begin{proof} 
	If $a=a^u$, then there is nothing to prove, so suppose otherwise. Let $b$ and $c$ be the neighbors of $a$ on $\gamma_u$. By Lemma~\ref{lemma:2} we have that $b \sim a^u \sim c$.  As $a$ has defect $1$ it has a unique neighbor $d$ in the interior of $S$ and $b\sim d\sim c$. 
	But then there is a 4--cycle on the vertices $(b,d,c,a^u)$ which has to have a diagonal. However, $b$ and $c$ are not adjacent, as $\gamma_u$ is a geodesic. But then  $a^u$ must be adjacent to $d$. 
\end{proof}

\begin{lemma}[nonadjacency] \label{lemma:4}
	Suppose $\Gamma$ acts {without stabilizing a simplex} on a systolic complex $X$. Let $S$ be the surface constructed in \cref{sec:surfaceConstruction}. Fix a vertex $u \in \{r,s,t\}$ and suppose there exist two adjacent vertices $a$ and $b$ on $\gamma_u$, where $a$ is an inner vertex of defect 1. Let $c$  be the unique neighbor of $a$ in $S$ not contained in $\gamma_u$. Then $b$ is not adjacent to $c^u$ and $c$ is not adjacent to $b^u$. In particular $b\neq b^u$. 
\end{lemma}
\begin{proof}
	Note first that $c$ is not adjacent to $c^u$ by minimality of the area of $S$. Let $d$ be the neighbor of $a$ other than $b$ on $\gamma_u$ (which exists as $a$ is an inner vertex), and assume that $c$ is adjacent to $b^u$. As the action is simplicial this is equivalent to the case where $b$ is adjacent to $c^u$. Then the path $(d,c,b,c^u,d^u)$ forms a 4- or a 5-cycle (depending on whether $d=d^u$ or not). In each case, $b$ is adjacent with neither $d$ by minimality of the length of $\gamma_u$ and also not adjacent to $d^u$ by Lemma~\ref{lem:u-stableSplx}. Hence it would follow that $c$ is adjacent with $c^u$, which is a contradiction.
\end{proof}


\section{The surface $S$ is not a $2$-simplex}\label{sec:notasimplex}

In this section we will show, see \cref{prop:notasimplex}, that for an action without a $\Gamma$- invariant simplex the surface $S$ will not consist of a single simplex. 

For the rest of this section suppose that $\Gamma$ acts without stabilizing a simplex. 
Thus \cref{prop:existence} implies that there is a nondegenerate surface $S$ satisfying the properties of \cref{sec:surfaceConstruction}. We will use the Notation as introduced in  \cref{notation:triangle-groups} in particular.

\begin{lemma}[Existence of many edges]\label{lem:2-simplex_2_help1}
	If $S$ is a $2$-simplex, then $a \sim b$ for all vertices $a \in x^{\langle r,s\rangle }$ and $b \in y^{\langle r,s\rangle }\cup z^{\langle r,s\rangle }$. 
\end{lemma}
\begin{proof}
	By Lemma~\ref{lemma:2}, $z \sim x^r$ and $y \sim x^s$ and, since $r$ and $s$ commute, $x^{\langle r,s\rangle }$ spans a simplex which is stabilized by $r$ and $s$ according to Lemma~\ref{lem:commuting uv}. Thus $X$ contains the closed path $(z,x^r,x^s,y)$. Then either $z \sim x^s$ or $y \sim x^r$ by $6$-largeness. Without loss of generality we assume that $z \sim x^s$. Then $X$ contains the closed path $C'=(z^r,z,x^s,x^{sr})$ and hence $x^s\sim z^r$ and $ z \sim x^{rs}=x^{sr}$. Therefore $z \sim a$ for all $a\in x^{\langle r,s\rangle }$. The action of $\Gamma$ is simplicial which implies that $a \sim b$ for all $a \in x^{\langle r,s\rangle }$ and $b \in z^{\langle r,s\rangle }$.
	
	The closed path $C=(y,z,x^{rs},z^s,y^s)$ in $X$ is either a $4$-- or a $5$--cycle depending on whether $y = y^s$ or not.
	We show that it has the diagonal $(y^s,x^{rs})$. Note that $z \nsim z^s$ since otherwise $z$ would be contained in $X_s\cap X_t \cap X_r$ and $S$ would not be minimal. If $C$ has length $4$ it follows by $6$-largeness that $y = y^s \sim x^{rs}$. If $C$ has length $5$ and $(y^s,x^{rs})$ is not in $X$, then $C$ has the diagonal $(y^s,z)$, as otherwise $(y^s,z,x^{rs},z^s)$ would form a $4$-cycle without diagonals. nonexistence of $(y^s,z)$ implies nonexistence of the diagonal $(y,z^s)$. Thus $C$ would contains at most one diagonal which contradicts $6$-largeness. 
	
	We conclude that $X$ contains the closed path $(y^s,x^{rs},x^r,y)$ and by $6$-largeness follows that $y^s \sim x^r$ and $y \sim x^{rs}$. Furthermore $y \sim x^s$ by Lemma~\ref{lemma:2}. Hence $y \sim a$ where $a\in x^{\langle r,s\rangle }$. Since $r$ and $s$ commute, $x^{\langle r,s\rangle }$ spans a simplex stabilized by $r$ and $s$ by Lemma~\ref{lem:commuting uv}. The action of $\Gamma$ is simplicial, therefore $a \sim b$ for all $a \in x^{\langle r,s\rangle }$ and $b \in y^{\langle r,s\rangle }$.
\end{proof}

We show that we can reduce the considerations to the case in which $y^{\langle s,t\rangle }$ and  $z^{\langle r,t\rangle }$ span a simplex. 

\begin{lemma}[Existence of a simplex]\label{lem:simplexfall}
	If $S$ is a $2$-simplex, then we may assume that $y^{\langle s,t\rangle }$ and $z^{\langle r,t\rangle }$ span a simplex.
\end{lemma}
\begin{proof}
	We assume that $z^{\langle r,t\rangle }$ does not span a simplex. Then there exists by Proposition~\ref{prop:H} a vertex $\bar z \in X_r \cap X_t$ which is adjacent to $z$ so that $\bar z^{\langle r,t\rangle }$ spans a simplex. We will show that $x$ and $y$ are adjacent to $\bar z$. Then $x$, $y$ and $\bar z$ span a surface $S'$ with the same minimality properties as $S$ and $\bar z^{\langle r,t\rangle }$ spans a simplex. If $y^{\langle r,t\rangle }$ spans a simplex, $S'$ is a surface we are looking for. Otherwise repeat the arguments after replacing $z$ with $y$ and $r$ with $s$. We  then obtain a surface $S''$ spanned by $x$, $\bar y$ and $\bar z$ with the same minimality properties as $S'$ in which $\bar{y}^{\langle r,t\rangle }$ spans a simplex. As the vertex $\bar z$ does not change, also $\bar z^{\langle r,t\rangle }$ spans a simplex in $S''$ and $S''$ is a surface we are looking for. 
	
	First we will show that $x\sim \bar z$. Suppose that $x$ is not adjacent to $\bar z$. 
	
	Lemma~\ref{lemma:2} implies that $y\sim z^t$ and from Proposition~\ref{prop:H} one can conclude that  $\bar z$ is adjacent to $z^t$ and $z^r$. Hence $X$ contains the $5$-cycle $C: =(x,y,z^t,\bar z,z^r)$.  By  Proposition~\ref{prop:H}, $z^r$ and $z^t$ are not adjacent. As $x \nsim \bar z$, $x$ is not adjacent to $z^t$ as otherwise $(x,z^t,\bar z,z^r)$ would be an $4$-cycle without diagonals. Hence the remaining diagonals  $(y,z^r)$ and $(y,\bar z)$ of $C$ are contained in $X$. 
	
	We now construct a $5$--cycle $C'$. We have shown that $y\sim \bar z$. By Proposition~\ref{prop:H}, $\bar z \sim z^{tr}$. We have seen that $y \sim z^t$  and hence $z^{tr}\sim y^r$.  Lemma~\ref{lem:2-simplex_2_help1} implies that  $x\sim y^r$. 
	One may then conclude that the complex $X$ contains the $5$-cycle $C':=(x,y,\bar z,z^{tr},y^r)$.
	Recall that $y \nsim y^r$ and $x \nsim \bar z$ by assumption. Hence $\bar z \nsim y^r$ as otherwise $(y^r,x,y,\bar z)$ would form a $4$-cycle without diagonals. The remaining two diagonals $(x,z^{tr})$ and $(y,z^{tr})$ of $C'$ are contained in $X$. Then $z^t\sim y^r$ and $(x,y,z^t,r^y)$ forms a $4$-cycle without diagonals which is a contradiction.  Therefore $x\sim \bar z$. 
	
	It now remains to show that $y\sim \bar z$. Lemma~\ref{lemma:2} implies $y\sim z^t$ and by  Proposition~\ref{prop:H} one has $\bar z\sim z^t$ and $z \sim z^r$. As $x\sim \bar z$, $X$ contains the $4$-cycle $(x,y,z^t,\bar z)$. If it contains the diagonal $(y,\bar z)$, we are done. Otherwise it contains the diagonal $(x,z^t)$ by $6$-largeness. 
	Then $(x,z^t,\bar z^t,z^r)$ forms a $4$-cycle. By Lemma~\ref{lemma:2} the vertex $x$ is adjacent to $z^r$ and $z^t\sim \bar z^t$. By Proposition~\ref{prop:H} we have  $\bar z^t\sim z^r$.  Furthermore Proposition~\ref{prop:H} implies that this cycle does not contain the diagonal $(z^r \nsim z^t)$. Hence $x \sim \bar z^t$ and also $x^t\sim \bar z$. 
	
	Then $X$ contains the closed path $(x,y,y^t,x^t,\bar z)$ which does not contain the diagonal $(x,x^t)$ as otherwise $x\in X_t$. If $x \nsim y^t$ the $5$--cycle $(x,y,y^t, x^t, \bar z)$ has to have the diagonals $(y, \bar z)$ and $(y^t, \bar z)$.  If $x \sim y^t$ the complex $X$ contains the $4$-cycle $(x,y,x^t,\bar z)$ and again $y$ is adjacent to $\bar z$. 
\end{proof}

\begin{lemma}[Existence of edges]\label{lem:2-simplex_2_help2}
	If $S$ is a $2$-simplex and $X$ contains the edge $(y^s,z)$ then $y^{s}\sim z^{t}$.  
\end{lemma}
\begin{proof} 	
	Observe that by Lemma~\ref{lem:simplexfall} the orbit $y^{\langle s,t\rangle }$ spans a simplex. Since $y^{s}\sim z$ in $X$, there exists the closed path $(z,y^{s},y^{st},z^{t})$. If this path is not a $4$-cycle, we are done. Otherwise it contains one of the two possible diagonals by $6$-largeness. But then $y^{s} \sim z^{t}$.
\end{proof}

\begin{lemma}[Existence of more edges]\label{lem:2-simplex_help3}
	If $S$ is a $2$-simplex and $X$ contains the edge $(x,z^{rt})$, then $y \sim z^{r}$.   
\end{lemma}
\begin{proof}
	Let $(x,z^{rt})$ be contained in $X$.
	First we consider the case where $x \nsim y^{t}$. By Lemma~\ref{lemma:2}, $x \sim z^{r}$ and thus $x^{t}~\sim z^{rt}$. Hence $X$ contains the closed path $(x,y,y^{t},x^{t},z^{rt})$. It is a $4$- or $5$-cycle depending on whether $y = y^{t}$ or not. We consider the more difficult case where $y \ne y^{t}$. Then $x \nsim y^{t}$ by assumption and $x \nsim x^{t}$ since otherwise $x \in X_r\cap X_s \cap X_t$. But this contradicts minimality of $S$. By $6$-largeness the cycle has the two diagonals $(y,z^{rt})$ and $(y^{t},z^{rt})$ which implies $y\sim z^{r}$ using the $\Gamma$-action. 
	
	Consider now the remaining case where  $x \sim y^{t}$. Recall that $x^{t} \sim z^{rt}$ by Lemma~\ref{lemma:2}. Hence $X$ contains the $4$-cycle $(x,y^t,x^t,z^{rt})$. Since $x \nsim x^{t}$ it contains by $6$-largeness the diagonal $(y^{t},z^{rt})$ and hence $y \sim z^{r}$. 
\end{proof}

\begin{lemma}[nonexistence of edges]\label{lem:2-simplex_help4}
	If $S$ is a $2$-simplex and $X$ does not contain the edge $(y,z^{r})$, then $x \nsim y^t$.    
\end{lemma}
\begin{proof}
	Observe that by Lemma~\ref{lem:simplexfall} the orbit $z^{\langle r,t\rangle }$ spans a simplex.  
	Suppose for a contradiction that $x \sim y^t$. By Lemma~\ref{lemma:2} the vertex $x\sim z^{r}$ and thus $x^t \sim z^{rt}$. The complex  $X$ then contains the closed path $C=(x,y^t,x^t,z^{rt},z^{r})$. It is a $4$- or $5$-cycle depending on whether $z^{r} = z^{rt}$ or not. We consider the most difficult case in which it is a $5$-cycle. Clearly it does not contain the diagonal $(x,x^t)$ since otherwise $x \in X_s\cap X_t\cap X_r$. It is $y^t \nsim z^{rt}$ by assumption. Furthermore it does not contain the diagonal $(x,z^{rt})$ because otherwise $(x,z^{rt},x^t,y^t)$ would form a $4$-cycle without diagonals. Using the $\Gamma$-action we conclude that $x^t\nsim z^{r}$. 
	But then $C$ contains at most $1$ diagonal which contradicts $6$-largeness. 
\end{proof}

We are now ready to prove the main result in this section saying that $S$ contains more than just a single $2$-simplex. 

\begin{prop}[$S$ is not a 2-simplex]\label{prop:notasimplex}
	If $\Gamma$ acts {without stabilizing a simplex} on a systolic complex $X$ the surface $S$ is not a $2$-simplex.
\end{prop}

\begin{proof} Recall that in the given situation  $X_r\cap X_s\cap X_t =\emptyset$. 
	We prove the proposition by contradiction arriving at statements that either  contradict $6$-largeness or the fact that $X_r\cap X_s\cap X_t\neq\emptyset$. 	So suppose that $S$ is a $2$-simplex.
	
	By Lemma~\ref{lem:simplexfall} we know that $y^{\langle s,t\rangle }$ and $z^{\langle r,t\rangle }$ span a simplex. First we show that $X$ contains at least one of the two edges $(y,z^{r})$ and $(y,z^s)$. Second, we show that $X$ contains exactly one of these two edges. 
	
	We assume that one of the two edges exists and arrive at a contradiction to $6$-largeness. Hence $S$ is not a single $2$-simplex. 
	
	\emph{Claim 1: $X$ contains either $(y,z^s)$ or $(y, z^r)$.}\\
	\noindent
	Suppose for a contradiction that $X$ does neither contain $(y,z^s)$ nor $(y, z^r)$. Then the $4$-cycle $(x,z^r,z^t,y)$ has the diagonal $(x \sim z^t)$ by $6$-largeness. Then $X$ contains the closed path $(x,y^s,y^{st},x^t,z^t)$.  Lemma~\ref{lemma:2} implies that $x\sim y^s$  and that $y^s\sim y^{st}$ as $y^{\langle s,t\rangle }$ spans a simplex. Using the action we see that $y^{st}\sim x^t$ and $x^t\sim z^t$. By assumption $x\nsim x^t$ and $z \nsim y^s$ and we obtain that $z^t \nsim y^{st}$. Then $x \nsim y^{st}$ as otherwise $(x,z^t,x^t,y^{st})$ would form an $4$-cycle without diagonals.  Hence the remaining two diagonals are contained in $X$ by $6$-largeness. In particular $x^t\sim y^s$. But then $(x,y^s,x^t,z)$ forms a $4$-cycle without diagonals. This is a contradiction.  Hence $X$ contains either $(y,z^s)$ or $(y, z^r)$.
	
	\emph{Claim 2: $X$ contains exactly one of the two edges $(y,z^s)$ and $(y, z^r)$}.\\
	\noindent 
	We have proven already that at least one of both edges is contained in $X$. Hence it remains to prove that both edges are not contained simultaneously. Assume that $X$ contains both edges $(y,z^s)$ and $(y, z^{r})$.    
	The existence of  $(y,z^s)$ and $(y, z^{r})$ implies the existence of the edges $(z^s,y^{rs})$ and $(y^{rs},z^{r})$ as $r$ and $s$ commute. Hence $X$ contains the closed path $(y,z^s,y^{rs},z^{r})$ which is a $4$-cycle. 
	By $6$-largeness it contains a diagonal.
	If $y\sim y^{rs}$ then $y^{s} \sim y^r$ and $X$ contains the  $4$-cycle $(y^{sr},y,y^{s},y^r)$. 	Since $s$ and $r$ commute both diagonals exists which implies that $y \sim y^r$. But this is impossible since then $S$ would not be minimal. Thus the cycle has no  diagonals which contradicts $6$-largeness. Analogously if $z^s\sim z^{r}$  then $X$ contains the 4-cycle $(z^{rs},z,z^{r},z^s)$ without diagonals an we have arrived at  contradiction. 
	
	We are now ready to prove the main assertion.  
	
	By Claim $2$, the complex $X$ contains exactly one of the two edges $(y,z^s)$ and $(y,z^r)$. For symmetrical reasons we may assume that $X$ contains the edge $(y,z^s)$, but not the edge $(y,z^r)$. 
	As $z^{\langle r,t\rangle }$ and $y^{\langle s,t\rangle }$ span a simplex, $y^s = y^{tst}$ or $y^s \sim y^{tst}$. By Lemma~\ref{lem:2-simplex_2_help1}, $X$ contains the edge $(x,z^s)$ and by Lemma~\ref{lem:2-simplex_2_help2}, $X$ contains the edge $(y^s,z^{t})$. Furthermore $z \sim y^t$ by Lemma~\ref{lemma:2}. Thus $X$ contains the closed path $C_1=(y^s,z^{t}, x^t,z^{st},y^{tst})$. A case by case  analysis shows that $C_1$ is a cycle of length $4$ or $5$ depending on whether $y^s = y^{tst}$ or not. First we show that $C_1$ does not contain the diagonal $(y^s,x^t)$. 
	To arrive a contradiction we assume that $C_1$ contains the diagonal $(y^{s},x^t)$. Then  $X$ contains the closed path $C_2=(x,y^{s},x^t,z^{rt},z^{r})$, because of Lemma~\ref{lem:2-simplex_2_help1} and as $z^{\langle r,t\rangle }$ is a simplex. Observe that it is a $5$-cycle. 
	
	We show that $C_2$ contains the diagonal $(y^{s}, z^{r})$. Assume that this is not the case. By assumption $C_2$ does not contain the diagonal $(x, x^t)$. Thus it does not contain  the diagonal $(z^{r},x^t)$ since otherwise $(x,z^{r},x^t,y^{s})$ would be a $4$-cycle without a diagonal. So $C_2$ contains the remaining two diagonals. In particular $x \sim z^{rt}$. This contradicts Lemma~\ref{lem:2-simplex_help3}. Hence $C_2$ contains the diagonal $(y^{s}, z^{r})$. But then $X$ contains the closed path $C_3=(y^{s},z^{s},z^{rs},y^{rs},z^{r})$, as $s$ and $r$ commute and $X$ contains the edge $(z,y^s)$. 
	
	A case analysis shows that the length of $C_3$ is $4$ or $5$ depending on whether $z = z^{r}$ or not. If the length is only $4$, there exists the diagonal $(y^{s},y^{rs})$ or $(z^{r},z^{rs})$ which both leads to a contradiction. Thus the cycle has length $5$. By assumption it does not contain the diagonals $(z^r,z^{rs})$, $(y^{s},y^{rs})$ and $(z^s,y^{rs})$. Thus it contains the remaining diagonals. In particular $y^s \sim z^{rs}$. But then $(y^s,z^r,y^{rs},z^{rs})$ forms a $4$-cycle without a diagonal. This contradicts $6$-largeness. 
	
	We have now shown that $C_1$ does not contain the diagonal $(y^s,x^t)$. Since $z\notin X_s$ the cycle $C_1$ does not contain $(z^t,z^{st})$. Then $y^s \nsim z^{st}$ since otherwise $(y^s,z^{st},x^t,z^t)$ would be a cycle of length $4$ without a diagonal. So $C_1$ has the two  diagonals $(z^t,y^{tst})$ and $(y^{tst},x^t)$. In particular $y\sim z^{st}$. Then $X$ contains the $4$-cycle $(y,z^{st},x^t,z^t)$ by Lemma~\ref{lem:2-simplex_2_help1} and Lemma~\ref{lemma:2}. The diagonal $(z^t,z^{st})$ does not exist by construction and by Lemma~\ref{lem:2-simplex_help4} one has $y\nsim x^t$. But then it is a $4$-cycle without diagonals which contradicts $6$-largeness. 
\end{proof}


\section{Defects at corners of $S$}\label{sec:def-corner}

In this section we study the defects on the corners of the minimal surface $S$. Notation is as in Section~\ref{sec:minimalSurface}. Note that not all the sides of $S$ need to contain inner vertices. There exists however, by \cref{prop:notasimplex} at least one side with at least one inner vertex. 

We assume that $\Gamma$ {acts without stabilizing a simplex} on a systolic complex $X$. \cref{prop:existence} implies then the existence of a nondegenerate minimal surface $S$ satisfying the hypotheses of Section~\ref{sec:minimalSurface}.

\subsection{Defects at any corner}\label{subsec:arbdef-corner}

The statements of this first subsection hold for arbitrary sides and corners of $S$. We use the following notation. 

\begin{notation}\label{notation2}
	Let $a$, $b$ and $c$ denote the three corners of $S$ and $u$, $v$ and $w$ the tree involutions generating $\Gamma$. Here we suppose that 
	\begin{itemize}
		\item $a$ is  the vertex in $X_u \cap X_w$,
		\item $b$ the vertex in $X_u \cap X_v$ and
		\item $c$ the vertex in $X_v \cap X_w$.
	\end{itemize} 
	We denote the geodesic sides of $S$ by $\gamma_u \subset X_u$, $\gamma_v \subset X_v$ and $\gamma_w\subset X_w$.
	So $\{a,b,c\}=\{x,y,z\}$ and $\{u,v,w\}=\{r,s,t\}$, but we do not specify the pairwise orders of the generators. 
	Furthermore denote by
	\begin{itemize}
		\item $a_u$, respectively $b_u$, be the neighbors of $a$, respectively $b$,  on $\gamma_u$, 
		\item $b_v$, respectively $c_v$, be the neighbors of $b$, respectively $c$, on $\gamma_v$ and let 
		\item $c_w$, respectively $a_w$, be the neighbors of $c$, respectively $a$, on $\gamma_w$, 
	\end{itemize}
	in case the respective sides have interior vertices. Note that it is possible that $a_u=b_u$, $b_v=c_v$ and $c_w=a_w$.
	
	Figure~\ref{fig:FlaecheS_konf} summarizes these choices and should serve as a quick reminder for how we named the various vertices.
	
	\begin{figure}[h!]
		\centering
		\includegraphics[width=4cm]{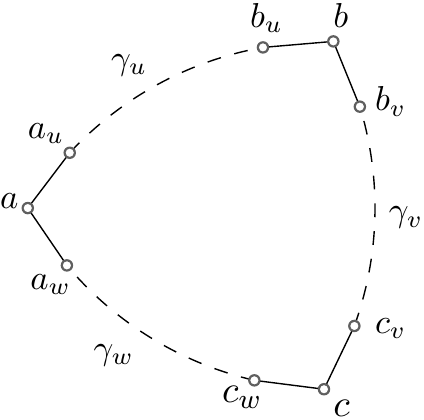}
		\caption{Notation for the surface $S$ in this section. }
		\label{fig:FlaecheS_konf}
	\end{figure}
\end{notation}

We may directly establish some upper bounds on defects. 

\begin{lemma}[Defect bounded by 2]\label{lemma:defects11}
	With notation as in~\ref{notation2} we have: 
	\begin{enumerate}
		\item The defect of any of the corners of $S$ is at most $2$.
		\item Suppose that both $\gamma_u$ and  $\gamma_w$ have at least one inner vertex and that $\defect(a_u)=\defect(a_w)=1$. Then $\defect(a)\leq 1$. 
	\end{enumerate}
\end{lemma}
\begin{proof}
	The first item follows from the minimality of the circumference of $S$.  
	To see the second item suppose that $a$ has defect two. Then $a_u$ and $a_w$ are connected by an edge. In addition the defect of both $a_u$ and $a_w$ equals one, hence they are connected to a common vertex $d \ne a$. We may conclude that then the vertex $d$ is in $X_u \cap X_w$ and may then replace $a$ by $d$ and shorten $\gamma_u$ and $\gamma_w$ contradicting the minimality of $S$. 
\end{proof}

The next two lemmata may seem a bit random. However, these situations will arise naturally in later proofs. 

\begin{lemma}[Defect bounded by 1] \label{lemma:help}
	Suppose $a^{\langle u,w\rangle}$ spans a simplex and $\gamma_u$ has at least one inner vertex. If in addition one of the following two conditions holds, then $\defect(a)\leq 1$.  
	
	\begin{enumerate}
		\item $\gamma_w$ has at least one inner vertex, $a_u^w\sim a_w$ and $a^u \nsim a_w$.  	
		\item $\gamma_w$ has no inner vertices, $a_u^w\sim c$ and $a^u \nsim c$.  
	\end{enumerate}
\end{lemma}
\begin{proof}
	Observe that in the second case, where $\gamma_w$ does not contain inner vertices, one has $a_w=c$. 
	We handle both cases simultaneously. Suppose for a contradiction that $a$ has defect $2$.
	The orbit $a^{\langle u,w\rangle}$ spans a simplex, by assumption $a_u^w \sim a_w$, and by \cref{lemma:2} we have that $(a^u,a^{uw}, a_u^w, a_w, a_u)$ forms a cycle. 
	Note that this cycle has at least length $4$ as $a_u^w \ne a_w \ne a_u$ by assumption and $a_u \ne a^u$ since otherwise $a^u=a_u\sim a_w$. 
	Minimality of $S$ implies that $a^u\nsim a_w$ and $a_u\nsim a_u^w$. Hence $a_u \sim a^{uw}$ by $6$-largeness. But then $(a_u,a^{uw},a_u^w,a_w^w)$ forms a 4-cycle, so either $a_u\sim a_u^w$ or $a^{uw} \sim a_w^w$ which both lead to contradictions. 
\end{proof} 

\begin{lemma}[Spanned simplex]\label{lemma:help2}
	Suppose we have $a_u \sim a_w^w$ in $S$ and that one of the following two conditions holds: 
	\begin{enumerate}
		\item $\defect(a)=2$,  $\defect(a_u)=1$ and $\defect(a_w)=0$ and both $\gamma_u$ and $\gamma_w$ have at least one inner vertex.
		\item $S$ consists of two $2$-simplices, the only inner vertex on $\gamma_u$ is $a_u$ and $X$ contains the edge $c^w \sim b$. 
	\end{enumerate}
	Assume further that $S$ contains the edge $(a_w,a_u^w)$, where $c = a_w$ in the second case.
	Then $S$ can be chosen so that $a^{\langle u,w\rangle}$ spans a simplex.
\end{lemma}

Note that the second case of \cref{lemma:help2} implies that $a_w = c$. 

\begin{proof}
	We handle both cases simultaneously. The surface $S$ as in item $1$ of the lemma is illustrated in Figure~\ref{fig:def10_b}\subref{help2b_SubfigA}. 
	Suppose that $a^{\langle u,w \rangle}$ does not span a simplex. In particular this implies that $\vert a^{\langle u,w \rangle} \vert \neq 4$ by $6$-largeness. Then Proposition~\ref{prop:H} implies that the orbit $a^H$ forms a cycle without diagonals and that there is a common neighbor $f \in X_u \cap X_w$ of $a$, $a^w$ and $a^u$ such that $f^{\langle u , w \rangle}$ is a simplex. Then $X$ contains the cycle $(f,a^u,a_u,a_w,a^w)$ by Lemma~\ref{lemma:2} which has diagonals. Lemma~\ref{lemma:4} implies that $a_w$ is not adjacent to $a^u$ and since the orbit $a^H$ forms a cycle without diagonals $a^u \ne a^w$. 
	So the cycle contains two of the three diagonals $(a_u,a^w)$, $(f,a_u)$ and $(f, a_w)$. If it contains the two diagonals $(f,a_u)$ and $(f, a_w)$, we obtain a new surface $S'$ with the same minimality properties as $S$ by replacing $a$ with $f$. Continue with this surface in place of $S$. 
	
	In the two remaining cases $a_u \sim a^w$. Then $X$ contains the cycle $(f,a^u,a_u,a^w)$ as illustrated in Figure~\ref{fig:def10_b}\subref{help2b_SubfigB}.  This cycle has to contain a diagonal and since the orbit $a^H$ forms a cycle without diagonals, $X$ contains the diagonal $(f,a_u)$. In particular $f^w \sim a_u^w$. By assumption $a_w\sim a_u^w$. Thus  $X$ contains the cycle $(f^w,a_u^w,a_w,a_u,f)$ as illustrated in Figure~\ref{fig:def10_b}\subref{help2b_SubfigC}. This cycle has to have diagonals and $a_u^w \nsim a_u$ by the minimality of $S$.  Hence either $a_w \sim f^w$, $a_u \sim f^w (\Leftrightarrow f \sim a_u^w)$ or $a_w \sim f$. If $a_w \sim f$, $X$ contains the two diagonals $(f,a_u)$ and $(f, a_w)$ and we obtain a new surface $S'$ with the same minimality properties as $S$ by replacing $a$ with $f$ like above.

	\subfiglabelskip=0pt
	\begin{figure}
		\centering
		\subfigure[][]{	\label{help2b_SubfigA}	
			\includegraphics{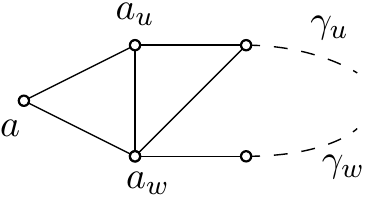}	}
		\subfigure[][]{\label{help2b_SubfigB}
			\includegraphics{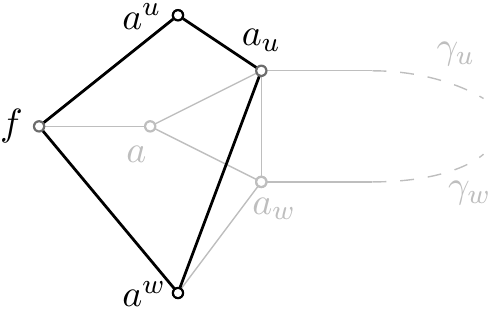}}	
		\subfigure[][]{\label{help2b_SubfigC}
			\includegraphics{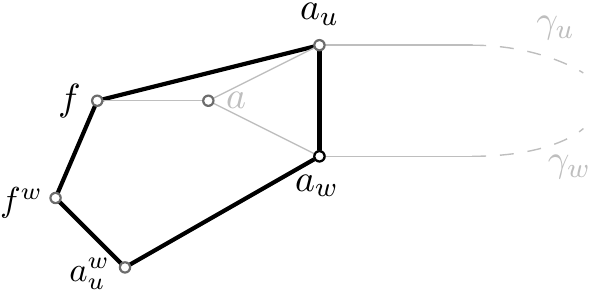}        }
		\caption{ A situation in the proof of Lemma~\ref{lemma:help2}. 
		}
		\label{fig:def10_b}
	\end{figure}

	
	The remaining case is that $a_u^w \sim f$. By assumption,  $a_w\sim a_u^w$ which implies that $a_w^w\sim a_u$ and $X$ contains the cycle  $(a_u^w,f,a_u,a_w^w)$ as illustrated in Figure~\ref{fig:def10_c}\subref{help2_SubfigA}. This cycle has to contain a diagonal. Minimality of the surface implies that $a_u \nsim a_u^w$ and that  $X$ contains the diagonal $(f, a_w^w)$.  In the second case $a_w=c$ and by assumption $c^w\sim b$. We then obtain a surface $S'$ with the same minimality properties as $S$ by exchanging $a$ with $f$ and $c$ with $c^w$. Otherwise $S$ satisfies the properties of the first item and $\gamma_u$ and $\gamma_w$ contain at least one inner vertex. Let $a_u'$ and $a_w'$ be the neighbors of $a_u$ and $a_w$ different from $a$ on the boundary of $S$ . Using the assumption on the defects of $a,a_u$ and $a_w$ we obtain that $a_u'\sim a_w'$. Moreover $a_w^w \sim a_w'$ by Lemma~\ref{lemma:2}. Hence $X$ contains the cycle $(a_u,a_w^w,a_w',a_u')$ as illustrated in Figure~\ref{fig:def10_c}\subref{help2_SubfigB}. Then  $6$-largeness implies that one of the diagonals $(a_u,a_w')$ and $(a_w^w,a_u')$ exists. In both cases we obtain a surface $S'$ with the same minimality properties as $S$ by exchanging $a$ with $f$ and $a_w$ with $a_w^w$ as illustrated in Figures~\ref{fig:def10_c}\subref{help2_SubfigC} and \ref{fig:def10_c}\subref{help2_SubfigD}.
\end{proof}

\subfiglabelskip=0pt
\begin{figure}
	
	\centering
	\subfigure[][]{	\label{help2_SubfigA}	
		\includegraphics{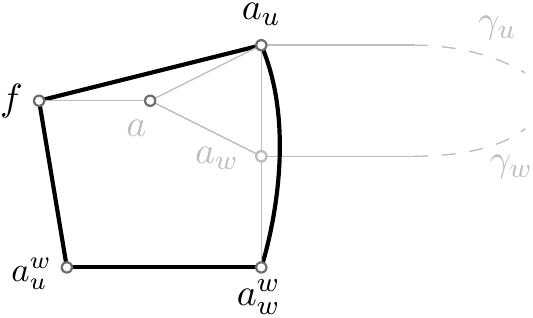}	}
	\subfigure[][]{\label{help2_SubfigB}
		\includegraphics{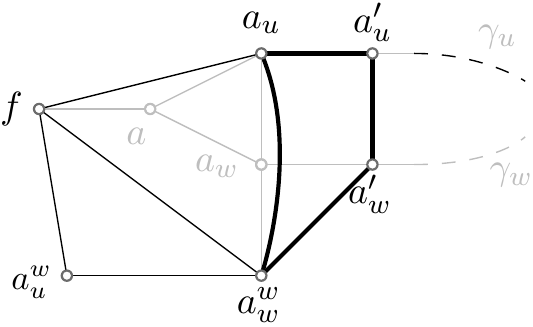}}	
	\subfigure[][]{\label{help2_SubfigC}
		\includegraphics{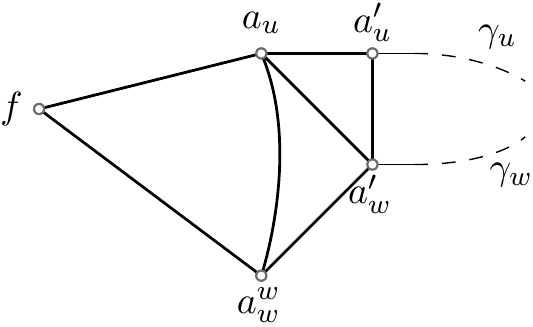}        }
	\subfigure[][]{\label{help2_SubfigD}
		\includegraphics{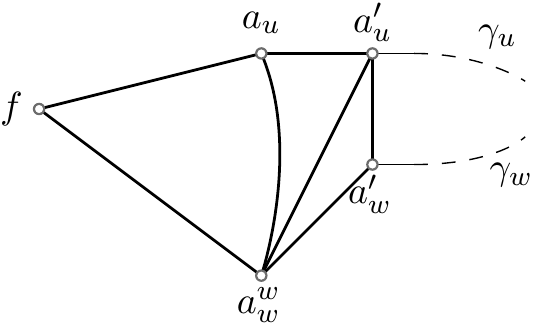}
	}	
	\caption{This figure illustrates the proof of Lemma~\ref{lemma:help2}. 
	}
	\label{fig:def10_c}
\end{figure}

\begin{lemma}[Bounding defects at a corner]\label{lemma:defectes10}
	Assume that both $\gamma_u$ and $\gamma_w$ have at least one inner vertex. If $a_u$ has defect $1$ and $a_w$ has defect $0$, then $a$ has defect  at most $1$.
\end{lemma}
\begin{proof}
	We prove the claim by contradiction. Assume that $a$ has defect $2$. Let $a_u'$ and $a_w'$ be the neighbors of $a_u$ and $a_w$ different from $a$ on $\gamma_u$ and $\gamma_w$ respectively.
	
	If $a_u$ is not adjacent to $a_w^w$, then the sequence $(a, a_w^w,a_w',a_u',a_u)$ forms a $5$-cycle by Lemma~\ref{lemma:2}. As $a$ is not adjacent to either $a_u'$ or $a_w'$ for minimality reasons, this $5$-cycle contains at least one of the two diagonals $(a_u,a_w')$ and $(a_w^w,a_u')$. In either case there occurs a $4$-cycle without a diagonal which is a contradiction. 
	So $a_u$ is adjacent to $a_w^w$. 
	
	By Lemma~\ref{lemma:help2} we may assume that $a^{\langle u,w \rangle}$ spans a simplex. 
	Then Lemma~\ref{lemma:4} implies that $a_u$ is not adjacent to $a_u^w$. But then $a$ has defect at most $1$ by Lemma~\ref{lemma:help} and we have reached a contradiction.
\end{proof}


\subsection{Defects at the corner of $S$ whose involutions commute   }\label{subsec:commdef-corner}

This subsection concerns the corner of $S$ whose involutions commute. Here we use the notation as in Section~\ref{sec:minimalSurface}. 

Note that the considered triangle groups contain two involutions which commute and recall from Notation~\ref{notation:triangle-groups} that we denoted them by $r$ and $s$. They correspond to two sides of $S$ which we denote by $\gamma_s$ and $\gamma_r$ with common corner $x$. Let $x_s$ be the neighbor of $x$ on $\gamma_s$ and $x_r$ be the neighbor of $x_r$ on $\gamma_r$. Notice that $x^r\sim x^s$ by 6-largeness.  The next three lemmas consider configurations at this special corner.

\begin{lemma}[Defects provided inner vertices]\label{lemma:commute1}
	With $r$ and $s$ the commuting involutions assume that both $\gamma_r$ and $\gamma_s$ have at least one inner vertex. If $\defect(x_r)=\defect(x_s)=1$ then $\defect(x)\leq 0$. 
\end{lemma}
\begin{proof}
	Suppose that $\defect(x_r)=\defect(x_s)=1$ and let $d$ be the unique neighbor of $x$ in the interior of $S$. 
	
	By \cref{lemma:4} we have that $d$ is neither adjacent to $x^s$ nor to $x^r$. Examining the 5-cycle $(x^s,x^r,x_r,d,x_s)$ we obtain that $x_s\sim x_r$.
	
	We repeat the same argument using \cref{lemma:3} with $x_s$ replaced by $x_s^s$ or with $x_r$ replaced by $x_r^r$ and obtain that $x_s \sim x_r \sim x_s^r \sim x_r^s \sim x_s$ (as $x_s^s \sim x_r^r$ implies $x_s^r = (x_s^s)^{sr} \sim (x_r^r)^{sr} = x_r^s$). This yields a 4-cycle which cannot have a diagonal as $x_s$ is not adjacent with $x_s^r$ and $x_r$ is not adjacent with $x_r^s$. Otherwise $x_s$ or $x_r$ would be contained in $X_s \cap X_r$, contradicting the minimality of $S$. This proves the lemma. \end{proof}

\begin{lemma}[Bounding defect]\label{lemma:commute3}
	Assume that both $\gamma_r$ and $\gamma_s$ have at least one inner vertex. If $\defect(x_s)=1$ and $\defect(x_r)=-1$ then $\defect(x)\leq 1$.
\end{lemma}
\begin{proof}
	A surface satisfying the conditions of the lemma is illustrated in Figure~\ref{fig:commute3}\subref{fig:def10_b_SubfigA}.
	We assume that $x$ has defect $2$ and show that this implies that $x_s^{s}\sim x_s^{r}$ or $x_r^{r}\sim x_r^{s}$. 
	If $x_s^{s}\sim x_s^{r}$, we conclude that $X$ contains the cycle $(x_s,x_s^s, x_s^r, x_s^{sr})$ using the action of $\Gamma$. This cycle is illustrated in Figure~\ref{fig:commute3}\subref{fig:def10_b_SubfigB}. By $6$-largeness this cycle has a diagonal. Since $r$ and $s$ commute, the existence of any diagonal implies that $x_s\sim x_s^r$. But then $x_s$ is contained in $X_s$ and $X_r$ and we can choose $x_s$ as corner of $S$ instead of $x$ yielding a surface with smaller area than $S$ which contradicts minimality of $S$. If $x_r^{r}\sim x_r^{s}$, we obtain by similar arguments that $x_r\sim x_r^s$, contradicting again minimality of $S$. 
	
	\subfiglabelskip=0pt
	\begin{figure}
		\centering
		\subfigure[][]{	\label{fig:def10_b_SubfigA}	
			\includegraphics{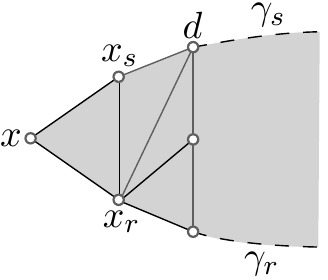}	}
		\subfigure[][]{\label{fig:def10_b_SubfigB}
			\includegraphics{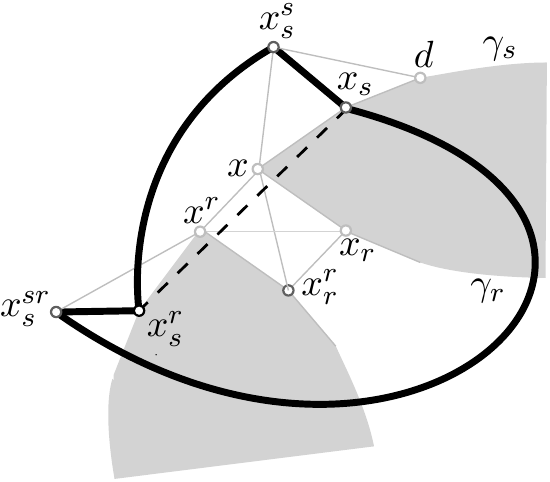}}	
		\caption{This illustrated the situation of Lemma~\ref{lemma:commute3}. The $4$-cycle in \subref{fig:def10_b_SubfigB} induces the existence of the dashed diagonal.}
		\label{fig:commute3}
	\end{figure}

	So it remains to show that $x_s^{s}\sim x_s^{r}$ or $x_r^{r}\sim x_r^{s}$. Let $d$ be the neighbor of $x_s$ different from $x$ on $\gamma_s$. 
	Notice that $x^{\langle r,s\rangle}$ has $4$ elements since $r$ and $s$ commute. Using $6$-largeness and the fact that $x\in X_s\cap X_r$ we obtain that $x^{\langle r,s\rangle }$ is a simplex. Thus $x^r\sim x^s$. By Lemma~\ref{lemma:2} the vertex $x_r \sim x^r$ and $x_s \sim x^s$. Since $x$ has defect $2$ the complex $X$ contains the cycle $(x^s,x^r,x_r,x_s)$.  By Lemma~\ref{lemma:4} this cycle does not have the diagonal $(x_r,x^s)$. Thus $x_s \sim x^r$ which implies that $x \sim x_s^r$. 
	
	Since $\defect(x_r)=-1$ and $\defect(x_s)=1$ we obtain that  $x_r\sim d$. Since $x_s$ and $x_r$ are both contained in $X_s$, Lemma~\ref{lemma:2} implies that $d \sim x_s^s$ and $x_s^s \sim x$. Thus $X$ contains the cycle $x \sim x_r \sim d \sim x_s^s$. Since $\gamma_s$ is a geodesic, this cycle does not have the diagonal $(x,d)$. Thus $x_r \sim x_s^s$. 
	
	We obtain moreover that $x_s^r \sim x_r^{rs}$, $x_r^s\sim x_r^{rs}$ and $x_r^s \sim x_s^s$ from the fact that $x_s \sim x_s^s$, $x_r \sim x_r^r$ and $x_r \sim x_s$. Since $x\in X_s$ we also conclude using Lemma~\ref{lemma:2} that $x \sim x_s^s$ and that $X$ contains the cycle $(x_r^{rs}, x_r^s,x_s^s,x,x_s^r)$. This cycle has to have diagonals. By Lemma~\ref{lemma:4} we obtain that $x_r\nsim x^s$ and thus this cycle does not contain the diagonal $(x,x_r^s)$. 
	There are thus four remaining possibilities for the diagonal. If $x_s^s \sim x_s^r$, we are done. If $x_r^s \sim x_s^r$, $X$ contains the cycle  $(x_r^{s}, x_s^r,x,x_s^s)$ and since $x \nsim x_r^s$ it follows $x_s^r \sim x_s^s$ and we are done. In the remaining two cases  $x_s \sim x_r^r$. If $x_s^s \sim x_r^{rs}$, then $x_s \sim x_r^r$ since $\Gamma$ operates simplicial. If $x \sim x_r^{rs}$, then $X$ contains the cycle  $(x_r^{rs}, x_r^s,x_s^s,x)$ and since $x \nsim x_r^s$ we conclude that $x_r^{rs}\sim x_s^s$ and thus $x_s \sim x_r^r$. 
	
	Using the $\Gamma$-action and the fact that $x_s \sim x_r^r$ we obtain a cycle $(x_s,x_r^r,x_s^{sr},x_r^s,x_s)$. But then $x_s^r \sim x_s^s$ or $x_r^r \sim x_r^s$ and we are done. 
\end{proof}


\section{Defects along the sides of $S$}\label{sec:def-sides-sum}

We assume in this section that $\Gamma$ acts {without stabilizing a simplex}  on a systolic complex $X$ and study defects along the sides of a nondegenerate minimal surface $S$ as constructed in \cref{sec:minimalSurface}. 
In the following we use notation as in \ref{notation2}. 
In particular, $\gamma_u$ denotes an arbitrary side of $S$. We will show that one can choose $S$ in such a way that the defect along any side is nonnegative. 

We summarize some first observations about the defects along a side obtained from simple counting arguments.   

\begin{lemma}[Counting defects]\label{lem:counting}
	The defect along any side $\gamma_u$ satisfies the following properties.  
	\begin{enumerate}
		\item If $\defect(a_u)<0$, then $\defect(\gamma_u) \leq 0$. \label{item_1}
		\item If $\defect(a_u)=-1$  and $\defect(\gamma_u)=0$, then the vertex closest to $b$ on $\gamma_u$ with nonzero defect has defect $1$. \label{item_2}
		\item If $\defect(a_u)=-1$ and $\defect(\gamma_u)=-1$, then the vertex closest to $b_u$ on $\gamma_u$ with nonzero defect has defect $-1$. \label{item_3}
		\item If $\defect(a_u)\le-2$, then $\defect(\gamma_u)\leq -1$. \label{item_4}
		\item If $\defect(a_u)=-2$ and $\defect(\gamma_u)=-1$, then the vertex closest to $b$ on $\gamma_u$ with nonzero defect has defect $1$.\label{item_5}
		\item If $\defect(a_u)<0$ and $\defect(b_u)<0$  and one of them has defect at most $-2$, then $\defect(\gamma_u)\leq -2$. \label{item_6}
		\item If $\defect(\gamma_u)=1$, the vertex on $\gamma_u$ closest to either end of $\gamma_u$ with nonzero defect has defect $1$.
		\label{item_7} 
	\end{enumerate}
\end{lemma}
\begin{proof} By Lemma~\ref{lem_sumdef}, any vertex on $\gamma_u$ has defect at most $1$. Furthermore two vertices of positive defect on $\gamma_u$ are separated by a vertex of negative defect. The claim follows by counting.  
\end{proof}

The key tool of this section is an edge swap, made precise in \cref{def:edgeswap}, which allows us to vary the surface $S$ by replacing two triangles forming a square that touches the boundary by a square on the same four vertices but with the other possible diagonal. Such a move will keep minimality of the surface intact while altering its defects on the boundary. The main goal is to prove that  there always exists a sequence of edge-swaps such that the resulting surface only contains sides of nonnegative defect. 

\begin{definition}[Edge-swaps and swap surfaces]\label{def:edgeswap}
	Let $S$ be a surface and $\gamma_u$ one of its sides.
	Let $p$ and $q$ be two adjacent inner vertices in $\gamma_u$. Let $m$ and $m'$ be two distinct vertices in $S \setminus \gamma_u$. Suppose that $(p,m,m',q)$ forms a $4$-cycle with diagonal $(m,q)$. 
	If a surface $S'$ differs from another surface $S$ by replacing the two simplices on $p,q,m$ and on $q,m,m'$ by the simplices on $p,m,m'$ and $p,m',q$, i.e.\ swapping the edge $(q, m)$ by $(p,m')$, we say $S'$ is obtained by an \emph{edge-swap along $\gamma_u$} from $S$.
	We call $S'$ a \emph{swap-surface of $S$ (along $\gamma_u$)} if $S'$ is obtained from  $S$ by a sequence of edge-swaps along (the same) $\gamma_u$.
	A \emph{repeated swap-surface $S'$ of $S$} is the end result of a sequence of swap-surfaces of  $S$ (along several sides), i.e.\ obtained by a sequence of edge-swaps which might be along changing sides. 
\end{definition}

Note that if a corner has defect 2 it may happen that an edge-swap along one of its incident sides simultaneously is an edge swap along the other incident side.

The following lemma shows  the existence of a swap-surface if two adjacent inner vertices of a side $\gamma_u$ have defect $1$ and defect $0$ respectively. 

\begin{lemma}[Existence of edge-swaps] \label{lemma:swaps}
	Let $p$ and $q$ be adjacent inner vertices on $\gamma_u$ of defect $1$ and $0$ respectively. Let $m$  be the unique neighbor of $p$ in $S$ not contained in $\gamma_u$ and let $m'$ be the neighbor of $q$ other than $m$ not contained in $\gamma_u$. Then there exists a surface $S'$ obtained by an edge-swap along $\gamma_u$. In particular, the  $1$-skeleton of $S'$ contains the edge $(p,m')$  and $S$ the edges $(q,m)$.
\end{lemma}
\begin{proof}
	Let $d\neq p$ be the neighbor of $q$ on $\gamma_u$. 
	Then $(p,m,m',d,q^u)$ forms a 5-cycle (by Lemma~\ref{lemma:2}) and hence has two diagonals. As $p$ is not adjacent to $d$ (as otherwise $\gamma_u$ would not be a geodesic), and $m$ not adjacent to $q^u$ by Lemma~\ref{lemma:4}, the only remaining possibility is that $q^u$ and $p$ are both  adjacent to $m'$. Thus the vertices $p,q,m$ and $m'$ span a 3-simplex $\Delta$.
	Two of the faces of this 3-simplex are triangles in $S$, namely $(p,q,m)$ and $(m,q,m')$ which can be replaced by the triangles $(p,q,m')$ and $(p,m,m')$ to obtain the desired surface $S'$. 
\end{proof}

Note that the following lemma in particular applies to the case where the defect along $\gamma_u$ is $1$. It will be used numerous times throughout the remainder of this section. 

\begin{lemma}[Moving defects with swaps]\label{lemma:shift}
	If the vertex closest to $a$ on $\gamma_u$ with nonzero defect has defect $1$, one can replace the surface $S$ by a surface $S'$ obtained by an edge-swap along $\gamma_u$ such that $a_u$ has defect $1$ in $S'$.  
\end{lemma}
\begin{proof}	
	Let $a_u'$ be the vertex  closest to $a$ on $\gamma_u$ with nonzero defect and $n$ its distance to $a_u$. if $n = 0$, we are done. Suppose that $n >0$. Let $\bar a_u$ be the vertex on $\gamma_u$ that has distance $n-1$ to $a_u$. By assumption, the defect of $a_u$ is zero and the defect of $\bar a_u$ is one. We apply Lemma~\ref{lemma:swaps} to these two vertices. This way, we obtain a new surface whose $1$-skeleton differs from $S$ by swapping two edges incident to $\bar a_u$ and $a_u'$. By construction, $a_u'$ has defect $0$ and $\bar a_u$ has defect $1$ in the new surface. By repeating this procedure $n$ times, we obtain a swap surface where all the swaps happened along $\gamma_u$ and in which $a_u$ has defect $1$. Each edge-swap  exchanges an edge $e_i$ for an edge $e_i'$  in $X$ such that the four endvertices of $e_i$ and $e_i'$ are contained in a $4$-cycle in the $1$-skeleton of $S$. Furthermore, either $e_i$ or $e_i'$ is contained in $S$. Hence, we can apply the $n$-th edge-swap to $S$. The resulting surface $S'$ is obtained by an edge-swap along $\gamma_u$ and $a_u$ has defect $1$. Hence $S'$ is the desired surface.
\end{proof}

\begin{lemma}[Effect of swaps on defects of edges]\label{lem:swap_othersides}
	If $S'$ is a surface obtained from $S$ by an edge-swap along $\gamma_u$, then the following is true. 
	\begin{enumerate}
		\item The defect of $\gamma_u$ in $S$ is the same as the defect of $\gamma_u$ in $S'$.
		\item The defects of $\gamma_w$ and $\gamma_v$ in $S'$ differ from their defects in $S$ by at most $1$. 
	\end{enumerate}
\end{lemma}
\begin{proof}
	Every edge-swap as in Lemma~\ref{lemma:swaps} changes the defects of the vertices of the involved edges on the boundary curve of $S$. Clearly the two vertices of the edge  in $\gamma_u$ are not contained in $\gamma_v\cup \gamma_w$. If the other two vertices are contained in the boundary of $S$, the defect of one of them increases and the defect of the other decreases by $1$. Thus the defect of $\gamma_w$ remains the same if it contains both vertices or none of them. It changes by $1$, if it contains one vertex. The same holds for $\gamma_v$. 
\end{proof}

\begin{lemma}[Effects of swaps on defects of corners] \label{lem:swap_corner}
	Suppose $S'$ is a  swap-surface of $S$ along $\gamma_u$. Then the defects of the corners $a,b$ on $\gamma_u$ in $S'$ are the same as in $S$. Corner $c$ has defect $2$ in $S$ if and only if it has defect $2$ in $S'$.
\end{lemma}
\begin{proof}
	Every edge-swap obtained by Lemma~\ref{lemma:swaps} changes the defects of the vertices of the swapped edges if they are on the boundary of $S$. By construction these vertices are not the corners associated with $\gamma_u$.  In particular the  defects of the corners incident to $\gamma_u$ do not change.
	If one of them correspond to the remaining corner not incident to $\gamma_u$, this corner has defect at most $1$.  	
\end{proof}

\begin{lemma}[Preserving defects]\label{lem:inner def1-vertex}
	If $v$ is an inner vertex of $\gamma_v$ or $\gamma_w$ of defect $1$, then its defect stays the same under any edge-swap along $\gamma_u$  if $v$ is not adjacent to a vertex $w$ of $\gamma_u$ with the following properties: 
	\begin{enumerate}
		\item $w$ has defect $0$ 
		\item $w$ is a neighbor of one of the corners incident to $\gamma_u$ and this corner has defect $2$.
	\end{enumerate}  
\end{lemma}
\begin{proof}
	Let $v$ be an inner vertex of defect $1$ not contained in $\gamma_u$ such that its defect changes by an edge-swap along $\gamma_u$. Then the defect of $v$ is contained in one of the swapped edges. Let $w$ be the second vertex of this edge. By definition it is contained in $\gamma_u$ and has defect $1$ or $0$. It has not defect $1$ as otherwise $w$ would be a corner of defect $1$ and would not be contained in an edge of the swap. Hence $w$ has defect $0$. Then $w$ is incident to exactly three  2-simplices, the vertex  $v$ is incident to exactly two $2$-simplices and $v$ and $w$ are adjacent. 
	Then $w$ is adjacent to a corner incident to $\gamma_u$ having defect $2$. 
\end{proof}

%
%

\begin{lemma}[Noncommuting involutions]\label{lem:def1-corner_commuting}
	If $\defect(c)=1$ and both vertices on the boundary of $S$ closest to $c$ with nonzero defect have defect $1$, then the involutions corresponding to $c$ do not commute. In particular, if a corner and its incident sides have defect $1$ the corresponding involutions do not commute. 
\end{lemma}
\begin{proof}
	Let $\gamma_w$ and $\gamma_v$ denote the sides incident with $c$. Then \cref{lemma:shift} implies that there exists a surface $S'$ obtained from $S$ via an edge-swap along $\gamma_w$ in which $\defect(c_w)=1$. By \cref{lem:swap_corner} the defect of $c$ does not change. We apply \cref{lemma:shift} again to $\gamma_v$ and obtain a surface $\hat S$ in which $c_v$ has defect $1$ by applying an edge-swap to $S'$ along $\gamma_v$. Using \cref{lem:swap_corner,lem:inner def1-vertex} the defects of $c$ and $c_w$ do not change. The fact that the involutions do not commute is then obtained from \cref{lemma:commute1}.
\end{proof}

\subfiglabelskip=0pt
\begin{figure}
	\centering
	\subfigure[][]{\label{commute3_SubfigA}\includegraphics[width=0.28\textwidth]{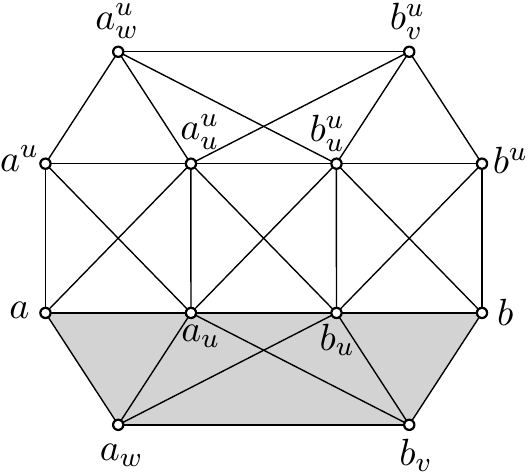}}
	\subfigure[][]{\label{commute3_SubfigB}\includegraphics[width=0.28\textwidth]{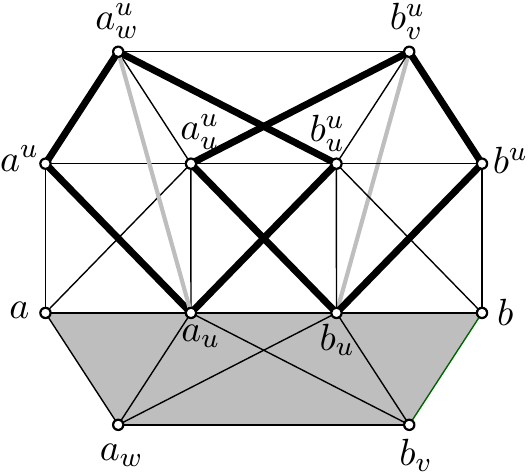}}	
	\subfigure[][]{\label{commute3_SubfigC}\includegraphics[width=0.28\textwidth]{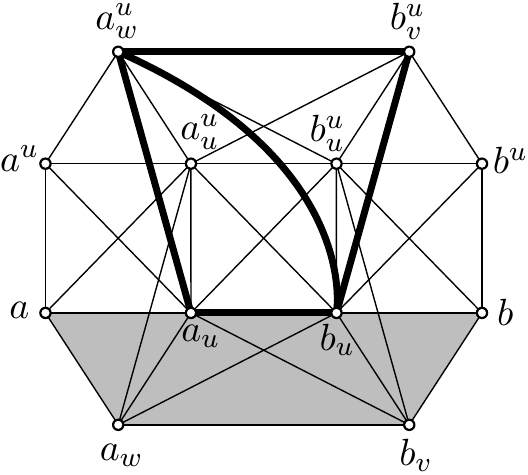}}	
	\subfigure[][]{\label{commute3_SubfigD}\includegraphics[width=0.28\textwidth]{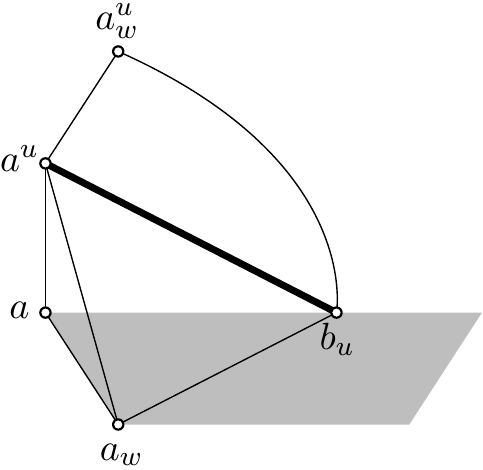}}	
	\subfigure[][]{\label{commute3_SubfigE}\includegraphics[width=0.28\textwidth]{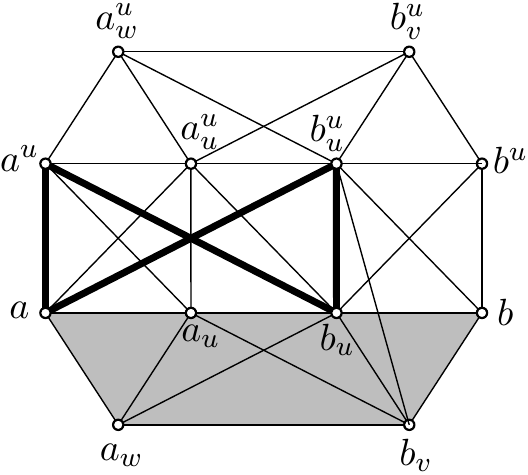}}	
	\caption{This illustrates the proof of Lemma~\ref{lem:def1_corner}. The last picture on the bottom shows the constructed $4$-cycle which does not have a diagonal.}
	\label{fig:defect_1_side}
\end{figure}

\begin{lemma}[Bounding defects of vertices]	\label{lem:def1_corner}
	Suppose that $\defect(\gamma_u)=1$ and that $\gamma_u$ has exactly $2$ inner vertices. Then $\min(\defect(a), \defect(b))\leq 1$. 
\end{lemma}
\begin{proof}
	We assume for a contradiction that the defect along $\gamma_u$ is $1$, $\gamma_u$ has exactly two inner vertices and that both incident corners $a$ and $b$ have defect $2$. 
	Then $\gamma_u$ contains just the vertices $a_u$ and $b_u$. One of them, say $a_u$, has defect $0$ and the other vertex $b_u$ has defect $1$. 
	
	By \cref{lemma:swaps}, we can replace $S$ by a surface in which $\defect(a_u)=1$ and $\defect(b_u)=0$ by swapping the edges $(a_u,b_u^u)$ and $(b_u,a_u^u)$. In particular, both edges are contained in $X$. Hence we are in the situation as illustrated in Figure~\ref{fig:defect_1_side}\subref{commute3_SubfigA}. 
	Note that $X$ contains the two $4$-cycles $(a^u,a_u, b_u^u,a_w^u)$ and $(a^u,a_u, b_u^u,a_w^u)$ which are marked by thick edges in Figure~\ref{fig:defect_1_side}\subref{commute3_SubfigB} By $6$-largeness, they contain a diagonal each. The edges $(a^u, b_u^u)$ and $(a_u^u,b^u)$ are not contained in $X$ as otherwise $a \sim b_u$ or $a_u \sim b$ using the simplicial action and this then implies that $\gamma_u$ is not a geodesic. Hence $X$ contains the edges $(a_u,a_w^u)$ and $(b_u,a_w^u)$ which are drawn in gray in Figure~\ref{fig:defect_1_side}\subref{commute3_SubfigB}. We conclude that $X$ contains the $4$-cycle $(a_u,b_u,b_v^u,a_w^u)$ (as illustrated in  Figure~\ref{fig:defect_1_side}\subref{commute3_SubfigC}) which, by $6$-largeness, contains a diagonal. 
	For symmetrical reasons we may assume that the same $4$-cycle also has the diagonal $(a_w^u,b_u)$. But then $X$ contains the closed path $C:=(a,a^u,a_w^u,b_u,a_w)$ which is shown in Figure~\ref{fig:defect_1_side}\subref{commute3_SubfigD}. It is a $4$-cycle or $5$-cycle depending on whether $a = a^u$ or not. 
	This cycle does not contain the diagonal $(a,b_u)$ as otherwise $\gamma_u$ would not be a geodesic. The vertices $a_w$ and $a_w^u$ are not adjacent as otherwise $a_w$ would be contained in $X_w$ and $X_u$ and $S$ would not have minimal area. Hence edge $(a, a_w^u)$ is not contained in $X$ as otherwise $(a,a_w^u,b_u,a_w)$ would be a $4$-cycle without diagonals. In particular $a \ne a^u$. We conclude that $C$ is a $5$-cycle that does not contain the diagonals $(a,a_w^u)$, $(a,b_u)$ and $(a_w^u,a_w)$. By $6$-largeness $C$ contains the remaining two diagonals $(a^u,b_u)$ and $(a^u,a_w)$. The edge $(a^u, b_u)$ is shown in Figure~\ref{fig:defect_1_side}\subref{commute3_SubfigD}. We observe that $X$ contains the closed path $(a,a^u,b_u,b_u^u)$ which is pictured in Figure~\ref{fig:defect_1_side}\subref{commute3_SubfigE}. Hence, $a \sim b_u$. This contradicts the fact that $\gamma_u$ is a geodesic. 
\end{proof}

\begin{lemma}[Number of inner vertices]\label{lem:2-def2corners}
	If every side of $S$ contains at least one inner vertex and $\defect(a)=\defect(b)=2$, then  $\gamma_u$ contains at least three inner vertices. 
\end{lemma}
\begin{proof}
	Suppose for a contradiction that $\gamma_u$ contains exactly two inner vertices. The defects of  $a$ and $b$ are at most $1$ by \cref{lem:def1_corner}. If $\gamma_u$ contains exactly one inner vertex $x$ we obtain that $x\sim a$ and $x\sim b$. As $\defect(a)=\defect(b)=2$ both $a$ and $b$ are contained in a single $2$-simplex each both of which contain $x$. As every side of $S$ contains at least one inner vertex these two $2$-simplices then can not be glued together along an edge. Hence $x$ cannot be incident to at least three $2$-simplices and $x$ cannot habe defect $1$ which contradicts what we have observed earlier.  
\end{proof}

The following lemma gives conditions for when one can shift two vertices of defect $1$ to the ends of the considered side via edge-swaps.

\begin{lemma}[Existence of swap surfaces]\label{lemma:defects_of_b}
	Suppose that $\defect(\gamma_u)=1$ in $S$ and that $\gamma_u$ contains at least $3$ inner vertices. Then there exists a swap surface $S'$ of $S$ along $\gamma_u$ such that $\defect(a_u)=\defect(b_u)=1$. Moreover, $\defect(\gamma_u)=1$ in $S'$.
\end{lemma}
\begin{proof}
	Let $a_u'$ be the vertex closest to $a$ on $\gamma_u$ that is not $0$. As the defect along $\gamma_u$ is $1$, $a_u'$ has defect $1$ by \cref{lem:counting} (\ref{item_7}). Hence we can apply an edge-swap and exchange the surface $S$ for a surface $S'$ via \cref{lemma:shift}. As $\gamma_u$ has at least $2$ inner vertices, none of the swapped edges is incident to $b_u$. Hence we can repeat the argument for the other corner of $\gamma_u$ and obtain the desired swap-surface. 
\end{proof}

Using these lemmas we are able to prove the following proposition.

\begin{prop}[Nonnegative defect on sides]\label{prop: sum_of_defects>=0}
	Given a nondegenerate minimal surface $S$ as constructed in \cref{sec:minimalSurface} there exists a surface $S'$ with the same minimality properties such that the defect along any side of $S'$ is $1$ or $0$.	
\end{prop}
\begin{proof} 
	Let $\gamma_u$ be a side of $S$. By Lemma~\ref{lem_sumdef} the defect along $\gamma_u$ is at most $1$. It remains to show that it is at least zero. Assume that the defect along $\gamma_u$ is less than $0$. By definition, any corner has defect at most $2$. \cref{prop:GaussBonnet} implies that the  sum of the defects along the boundary of $S$ is at least $6$. Thus there are only three cases we need to consider which we have illustrated in Figure~\ref{fig:Cases_lem_def>=0}. We use notation as in~\ref{notation2}.
	
	\begin{figure}[h!]
		\centering
		\includegraphics[]{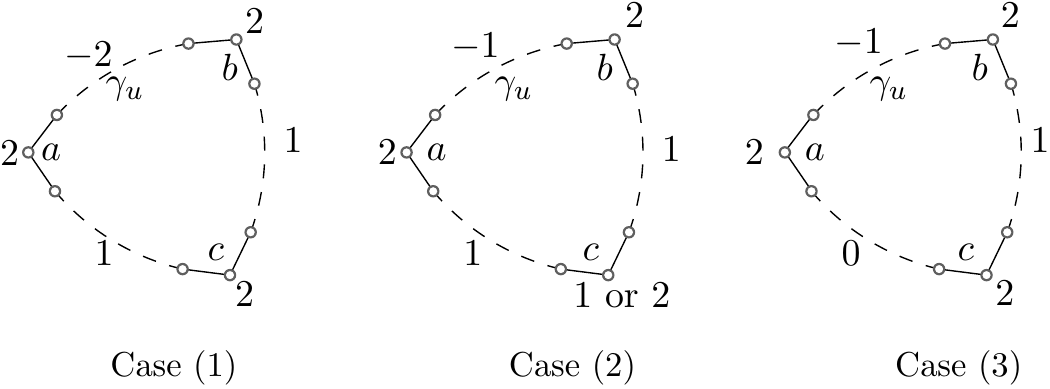}
		\caption{The three cases appearing in the proof of \cref{prop: sum_of_defects>=0} of surfaces with one side of negative defect. The numbers are the defects of the  sides and corners they label. }
		\label{fig:Cases_lem_def>=0}	
	\end{figure}

	\emph{{Case (1):} $\defect(\gamma_u)=-2$ and $\defect(\gamma_v)=\defect(\gamma_w)=1$.} \\
	\noindent
	In this case all corners of $S$ have to have defect $2$ as otherwise the defects along the boundary of $S$ sum up to at most $5$ which contradicts the combinatorial Gauss-Bonnet Lemma~\ref{prop:GaussBonnet}. By \cref{lemma:shift}, there is surface $S'$ obtained from $S$ by an edge-swap along $\gamma_w$ such that $c_w$ has defect $1$ in $S'$.  By Lemma~\ref{lem:swap_corner} the defects of the corners $a$, $b$, $c$ do not change. The vertex  $c_v$ has negative defect by Lemma~\ref{lemma:defectes10} and Lemma~\ref{lemma:defects11}. We conclude with help of \cref{lem:counting}(\ref{item_1}), that the defect  $\gamma_v$ in $S'$ is at most $0$. By Lemma~\ref{lem:swap_othersides}, the defect along $\gamma_v$ in $S'$ is $1$ or $0$. Hence it is equal to $0$. Furthermore  the defect along $\gamma_w$ is $1$ in $S'$ and the defect along $\gamma_u$ is $-3,-2$ or $-1$ in $S'$ by Lemma~\ref{lem:swap_othersides}. As the defects along the boundary of $S$ sum up to at least $6$ by the combinatorial Gauss-Bonnet Lemma~\ref{prop:GaussBonnet}, the defect along $\gamma_u$ is $-1$ in $S'$.  We will be dealing with the surface $S'$ in case (3).  
	
	\emph{{Case (2):} $\defect(\gamma_u)=-1$ and $\defect(\gamma_v)=\defect(\gamma_w)=1$.}\\
	\noindent
	Suppose first that $c$ has defect $2$. 
	As in case (1) we may argue that there is a  surface $S'$ obtained from $S$ via an edge-swap along $\gamma_w$ such that $\defect(c_w)=1$ and $\defect(c_v)=-1$ and where $\defect(\gamma_w)=1$. The defect along $\gamma_v$ is at most $0$ by \cref{lem:counting}(\ref{item_1}) and at least  $0$ by Lemma~\ref{lem:swap_othersides}. Thus $\defect(\gamma_v)=0$. Lemma~\ref{lem:swap_othersides} then implies that the defect along the third side in $S'$ is in $\{-2, -1, 0\}$. If the defect along $\gamma_u$ is $0$ we have found a desired surface. If it is $-1$ we have reduced the situation to the third case. 
	If it is $-2$ the defects along the boundary of $S$ sum up to at most $5$ which contradicts \cref{prop:GaussBonnet}. 
	
	Suppose now that $\defect(c)= 1$ in $S$. Then $a$ and $b$ have both defect $2$ and the involutions corresponding to $c$ do not commute by Lemma~\ref{lem:def1-corner_commuting} as  otherwise the defects along the boundary of $S$ sum up to at most $5$. Hence the involutions corresponding to $a$ or $b$ commute. We will show that this leads to a contradiction. 
	
	Recall that $\defect(\gamma_w)=1$ in $S$. By \cref{lemma:shift} there exists surface $S'$ obtained from $S$ via an edge-swap along $\gamma_w$ such that $a_w$ has defect $1$ in $S'$ and a another surface $\hat S$ obtained from $S'$ by an edge-swap along $\gamma_v$ such that $b_v$ has defect $1$. By \cref{lem:swap_corner,lem:inner def1-vertex} the defect of the corners do not change and $a_w$ and $b_v$ have defect $1$ in $\hat S$. The defects of both $a_u$ and $b_u$ are negative in $\hat S$ as otherwise the defect of the corresponding corners would be at most $1$ by \cref{lemma:defectes10,lemma:defects11}. 
	If either $\defect(a_u)$ or $\defect(b_u)$ is $<-1$ the defect along $\gamma_u$ is at most $-2$ by \cref{lem:counting}(\ref{item_6}). Then the defects along the boundary of $S$ sum up to at most $5$ which contradicts \cref{prop:GaussBonnet}. Hence $a_u$ and $b_u$ have both defect $-1$ in $\hat S$. In particular, $a$ and $b$ are both adjacent to a vertex of defect $1$ and a vertex of defect $-1$. But then neither the involutions corresponding to $a$ nor those corresponding to $b$ commute as otherwise the defect of $a$, respectively $b$, would be at most $1$ by \cref{lemma:commute3}. And we have arrived at a contradiction. 
	
	\emph{{Case (3):} $\defect(\gamma_u)=-1$, $\defect(\gamma_w)=0$ and $\defect(\gamma_v)=1$.}\\
	\noindent
	Suppose $S$ is a surface with the listed properties. Then each corner of $S$ has defect $2$, as otherwise the defects along the boundary of $S$ sum up to less then $6$. This implies that each side of $S$, also $\gamma_w$, has at least one inner vertex. If $\gamma_w$ does not contain an inner vertex $S$ is a single $2$-simplex. Then the defect along no side would be negative. 
	
	By Corollary~\ref{lem:2-def2corners} the side $\gamma_v$ contains at least $3$ inner vertices. Hence there exists a swap-surface $S'$ of $S$ along $\gamma_v$ such that $b_v$ and $c_v$ have defect $1$ and such that the defect along $\gamma_v$  is $1$ in $S'$ by Lemma~\ref{lemma:defects_of_b}. Every corner of $S'$ has defect $2$ by Lemma~\ref{lem:swap_corner}. Thus the vertices $b_u$ and $c_w$ have both negative defect by Lemma~\ref{lemma:defectes10} and Lemma~\ref{lemma:defects11}. By \cref{lem:counting}(\ref{item_1}), the defect along $\gamma_u$ and $\gamma_w$ is at most $0$. If both defects are $0$, we have found a desired surface. 
	
	In the remaining case one geodesic has defect $0$ and the other $-1$ as otherwise the sum of the defects along the boundary of $S$ is less than $6$. 
	Potentially switching the roles of $u$ and $v$ we may assume that $\gamma_u$ has defect $-1$ and $\gamma_w$ has defect $0$.
	
	Recall that $c_w$ has negative defect.  By \cref{lem_sumdef}, any vertex on $\gamma_u$ has defect at most $1$. Furthermore two vertices of positive defect on $\gamma_u$ are separated by a vertex of negative defect. Hence the vertex closest to $a_w$ on $\gamma_w$ with nonzero defect has defect $1$. By \cref{lemma:shift}  there exists a surface $S''$ obtained from $S'$ by an edge-swap along $\gamma_w$ such that $a_w$ has defect $1$ in $S'$.   From \cref{lem:swap_corner} we obtain that each corner of $S''$ has defect $2$. And Lemma~\ref{lem:inner def1-vertex} implies that the vertices $b_v$ and $c_v$ have defect $1$ in $S''$. It follows from   \cref{lem_sumdef}, that $\defect(\gamma_v)=1$ in $S''$. 
	
	Using \cref{lemma:defectes10,lemma:defects11} we may argue as above and obtain that the vertices $b_u$ and $c_w$ have both negative defect.  By \cref{lem:counting}(\ref{item_1}), the defect along $\gamma_u$ and $\gamma_w$ is at most $0$. If both defects are $0$, we have found a desired surface. Otherwise, $\defect(\gamma_w)=0$ in $S''$ by Lemma~\ref{lem:swap_othersides} and the defect along $\gamma_u$ is $-1$ because otherwise the defects along the boundary of $S''$ would not sum up to at least $6$. It follows from \cref{lem:counting}(\ref{item_1}) that the vertex $c_w$ has defect $-1$. 
	By \cref{lemma:defectes10,lemma:defects11} the defect of $a_u$ is negative. Recall that $c_u$ has negative defect and that $\gamma_u$ has defect $-1$. Hence \ref{lem_sumdef} implies that the defect of $b_u$ is $-1$. All in all,  $S''$ is a surface in which each corner is adjacent to a vertex of defect $-1$ and to a vertex of defect $1$. 
	But then also the corner whose corresponding involutions commute is adjacent to a vertex of defect $1$ and to a vertex of defect $-1$. This contradicts Lemma~\ref{lemma:commute3}.
\end{proof}

We have shown that $S$ can be chosen such that no side has negative defect. The next proposition and the lemma afterwards establish further properties of such surfaces. 

\begin{prop}[Bounding defects at corners]\label{prop:defectes_along_1_sides}
	With notation as in \ref{notation2} suppose the defect along all sides of $S$ is nonnegative and that $\gamma_u$ has defect one. Then $\defect(a)\leq 1$ and $\defect(b)\leq 1$. 
\end{prop}

\begin{proof}	
	We suppose that the defect along all sides of $S$ is nonnegative and that $\gamma_u$ has defect one. 
	We distinguish three cases.
	
	\emph{Case (1): There are no inner vertices in the sides $\gamma_v$ and $\gamma_w$.} \\  
	\noindent 	
	In this case, two corners have defect $2$ and the third corner has defect at least $1$ because the sum of the defects along the boundary of $S$ is at least $6$ by \cref{prop:GaussBonnet}. Because $\gamma_v$ and $\gamma_w$ don't have inner vertices and since $\gamma_u$ has at least one inner vertex, the corner $c$ does not have defect two. Hence, the corners $a$ and $b$ have defect two and $c$ has defect one. It follows that the corner $c$ is contained in exactly two $2$-simplices of $S$ and that $a$ and $b$ are contained in exactly one $2$-simplex respectively. As $\gamma_v$ and $\gamma_w$ don't have inner vertices, the sides $\gamma_v$ and $\gamma_w$ are contained in these two $2$-simplices respectively and we are in the special case illustrated in Figure~\ref{fig:specialcaseraute}.

	\begin{figure}[h!]
		\centering
		\includegraphics[]{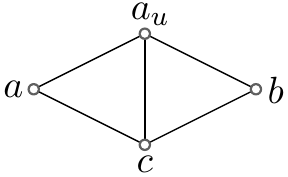}
		\caption{A special case of the surface $S$. }
		\label{fig:specialcaseraute}
	\end{figure}
	
	We prove that this special case pictured in Figure \ref{fig:specialcaseraute} does not occur. 
	Recall that one of the corners of $S$ correspond to two commuting involutions and that the orbit of this corner under the group generated by these two involutions spans a simplex by $6$-largeness. First we will show that neither the involutions corresponding to $a$ nor the involutions corresponding to $b$ commute. Therefore the commuting involutions have to correspond to $c$. We will see that this implies the existence of a $4$-cycle or a $5$-cycle without diagonals which contradicts $6$-largeness. 
	
	\emph{Claim 1: The involutions $u,w$ corresponding to $a$ and $b$ do not commute.}\\  
	\noindent 
	Assume otherwise. Then $C=(a^u,a_u,c,a^w)$ forms a closed path by Lemma~\ref{lemma:2}. By Lemma~\ref{lemma:4} the complex $X$ does not contain the edge $(a^u, c)$. So $C$ is a cycle of length $4$ containing the diagonal  $(a_u, a^w)$ by $6$-largeness. Hence $a\sim a_u^w$. Since $u$ and $w$ commute, $a_u^w \in X_u$ $(u(a_u^w)=w(a_u^u)\sim a_u^w$ because $a_u^u \sim a_u$) and $a \in X_u$ by definition. Thus $a^u\sim a_u^w$ by Lemma~\ref{lemma:2}. Applying Lemma~\ref{lemma:2} again and using that the group action is simplicial we conclude that $(a^u,a_u^w,c^w,c,a_u)$ forms a closed path. If $c = c^w$ it is a cycle of length  $4$ and contains $(a_u, a_u^w)$ or $(a^u, c)$ which both lead to a contradiction by the minimality of $S$ or Lemma~\ref{lemma:4}. If $c \ne c^w$, the cycle has length $5$ and the same argumentation yields that $a_u\nsim a_u^w$ and $a^u \nsim c$. Then the existence of the diagonal $(c,a_u^w)$ would lead to a $4$-cycle without diagonals. Hence $X$ does not contain $(c,a_u^w)$ and hence also not $(c^w,a_u)$. But then the described $5$-cycle has at most one diagonal which contradicts $6$-largeness.
	Thus $u$ and $w$ have to commute. Similarly one shows that the involutions corresponding to  $b$ do not commute. \qed\ claim 1.
	
	It remains to consider the subcase that the involutions $v,w$ corresponding to corner $c$ commute. We will use the following observation.
	
	\emph{Claim 2: If the involutions $v$ and $w$ corresponding to corner $c$ commute then $X$ contains the two diagonals $(a_u,c^v)$ and $(a_u,c^w)$ but not the diagonal $(b,c^w)$. }
	
	Since $v$ and $w$ commute, $X$ contains the closed path $(a,a_u,b,c^v,c^w)$. Using again that $v$ and $w$ commute, a case by case analysis shows that it is a cycle of length $4$ or $5$ depending on whether $c^v = c^w$ or not. By minimality of $\gamma_u$, $a$ is not incident to $b$. Thus either $a_u \sim c^w$ or $a_u \sim c^v$ since otherwise $X$ contains a $4$-cycle without a diagonal. This argument is symmetric hence we may assume that $a_u \sim c^w$. To arrive a contradiction assume further that $b \sim c^w$. Apply  Lemma~\ref{lemma:help2} and reduce the situation to the case that $a^{\langle u,w\rangle}$ spans a simplex as $a_u \sim c^w$, $a_u^w \sim c$. By Lemma~\ref{lemma:4} the complex $X$ does not contain the edge $(a^u, c)$. Hence $a$ has defect at most $1$ by Lemma~\ref{lemma:help} which has us arrived at a contradiction. Thus it remains to consider the case where $b \nsim c^w$. Then $(a_u,c^w,c^v,b)$ forms a $4$-cycle containing the diagonal $(a_u,c^v)$ by $6$-largeness. We have shown that one can reduce the situation to the case where $X$ contains the two diagonals  $(a_u,c^v)$ and $(a_u,c^w)$ but not the diagonals $(b,c^w)$. This proves the claim. \qed\  claim 2.
	
	Recall that it remains to consider the subcase that the involutions $v,w$ corresponding to corner $c$ commute. 
	By ~\cref{lemma:4}, the complex $X$ does not contain the edges $(a^u, c)$ and $(b^u,c)$. Hence we may apply Lemma~\ref{lemma:help} to the corners $a$ and $b$ and conclude that neither $a^{\langle u,w  \rangle}$ nor $b^{\langle u,v  \rangle}$ span a simplex. 
	By Proposition~\ref{prop:H} there exists a vertex $x\in X_v\cap X_u$ which is connected to $b^u$, $b$ and $b^v$  and $X$ contains the closed path $(a_u,b^u,x,b^v,c)$. It is not hard to see that this is a cycle of length $5$.
	By Proposition~\ref{prop:H} the orbit of $b$ under $\langle u,v\rangle $ spans a cycle without diagonals. Thus $b^u \nsim b^v$. By Lemma~\ref{lemma:4}, $c \nsim b^u$. Hence there are only two possible configurations: either $x$ is adjacent to $a_u$ and $c$ or $a_u \sim x$ and $a_u \sim b^v$. In the former case we obtain a new surface $S'$ with the same minimality properties like $S$ by exchanging $b$ with $x$. By Proposition~\ref{prop:H} we can choose $x$ so that $x^{\langle u,v\rangle}$ spans a simplex and we arrive a contradiction by applying Lemma~\ref{lemma:help} as above. In the second case $X$ contains the closed path $C=(c^w,a_u,b,a_u^v,c^{vw})$ which is a $4$- or $5$-cycle depending on whether $c=c^v$ or not. By assumption it does not contain the diagonal $(b,c^w)$ and it does not contain $(a_u,a_u^v)$ for minimality reasons. Notice that  $b^w \in X_v$ since $v(b^w)= w(b^v)\sim b^w$ as $b \sim b^v$. Thus $C$ does not contain $(b, c^{vw})$ as diagonal. Otherwise $b^w \sim c^v$  which leads to the existence of $(c,b^w)$ by applying Lemma~\ref{lem:u-stableSplx}. This contradicts  Claim 2.  Regardless of whether $X$ contains the remaining diagonals of $C$ or not, $X$ contains a $4$- or a $5$-cycle without a diagonal which contradicts $6$-largeness.

	\emph{{Case (2):} Exactly one of the two sides $\gamma_v$ and $\gamma_w$ does not have inner vertices}\\
	\noindent
	Potentially switching the roles of $v$ and $w$ we may assume that $\gamma_w$ has no inner vertices. 
	Suppose for a contradiction that one of the corners $a$ and $b$, say $a$, has defect two.
	\cref{lem:counting,lemma:shift} imply the existence of a surface $S'$ obtained from $S$ by an edge-swap along $\gamma_u$ such that the defect of $a_u$ is $1$. By \cref{lem:swap_corner} and \cref{lem:swap_othersides}, the defect along  $\gamma_u$ in $S'$ is $1$ and the defect of $a$ in $S'$ is $2$. Then $a_u$ and its neighbor $a_u' \ne a$ on $\gamma_u$ are both connected to $c$ in $S'$. If $a_u'=b$ the geodesic $\gamma_v$ does not contain inner vertices. Then $S'$ is a surface as in case (1) and we have proven already that such a surface does not exist. Hence we assume that $a_u'\ne b$. Then  $a_u'$ is an inner vertex of $\gamma_u$ that is connected to $c$ and $\gamma_v$ contains a further vertex adjacent to $c$ that neither coincides with $a$ nor with $a'$. Hence, the degree of $c$ in the $1$-skeleton of $S$ is at least $4$ and $\defect(c)\leq 0$. As the sum of the defects along the boundary of $S$ is at least $6$ by \cref{prop:GaussBonnet}., the defect along $\gamma_v$ is $1$ and  $\defect(b)=2$.
	We replace the surface $S'$ with a surface $S''$ that is obtained from $S'$ by an edge-swap along $\gamma_v$ such that $b_v$ has defect $1$ in $S''$  by means of \cref{lemma:shift}. By \cref{lem:swap_corner}, the defects of the corners of $S'$ stay the same. Then the defect of $b_u$ on $\gamma_u$ is negative in $S''$ by \cref{lemma:defectes10,lemma:defects11}.  Hence $\gamma_u$ has defect at most $0$ in $S''$ by \cref{lem:counting}(\ref{item_1}).  But then the sum of the defects along the boundary of $S''$ is at most $5$ which contradicts \cref{prop:GaussBonnet}.

	\emph{{Case (3):} Every side of $S$ has inner vertices.}\\
	\noindent 
	The proof in this case has to steps. In the first step we prove the claim that either $a$ or $b$ has defect at most $1$. In the second step we conclude that both corners have defect at most $1$.
	
	\emph{Claim 3: One corner incident to $\gamma_u$ has defect at most $1$.}\\
	\noindent
	To arrive a contradiction we assume that $a$ and $b$ have defect $2$ and show the existence of a repeated swap surfaces of $S$ which do not satisfy all necessary properties. 
	
	The side $\gamma_u$ contains at least $3$ inner vertices by Corollary~\ref{lem:2-def2corners}. Lemma~\ref{lemma:defects_of_b} implies that there exists a swap surface $S'$ of $S$ along $\gamma_u$ such that $a_u$ and $b_u$ have defect $1$ and such that the defect along $\gamma_u$ is $1$ in $S'$.  Then the vertices $a_w$ and $b_v$ have  negative defect in $S'$ as otherwise the corresponding corners would have defect at most $1$ by \cref{lemma:defectes10} and Lemma~\ref{lemma:defects11}. By \cref{lem:counting}(\ref{item_1}), the defect along $\gamma_w$ and $\gamma_v$ is at most $0$ in $S'$. At least one of them has defect $0$ as otherwise the sum of the defects along the boundary of $S$ is less than $6$. By symmetrical reasons we can assume without loss of generality that the defect along $\gamma_w$ is $0$. By \cref{lem:counting}(\ref{item_4}) and \cref{lem:counting}(\ref{item_2}), $a_w$ has defect $-1$ and the vertex closest to $c$ on $\gamma_w$ has defect $1$.  By \cref{lemma:shift}, there exists a surface $\hat S$ that is obtained from $S'$ by an edge-swap along $\gamma_w$ such that $c_w$ has defect $1$ in $\hat S$. By Lemma~\ref{lem:swap_corner} and Lemma~\ref{lem:inner def1-vertex}, the defect of the corners and the defects of  $a_u$ and $b_u$ do not change. As before we observe that $a_w$ and $b_v$ have negative defect in $\hat S$ and that the defect along $\gamma_w$ and $\gamma_v$ in $\hat S$ is at most $0$. 
	
	If corner $c$ has defect $2$ in $\hat S$ , $c_v$ has negative defect by Lemma~\ref{lemma:defectes10} and Lemma~\ref{lemma:defects11}. Then $b_v$ and $c_v$ have both negative defect. If one of these defects would be less than $-1$, the sum of the defects along the boundary of $S$ would be less than $6$. Hence any corner is adjacent to a vertex of defect $1$ and a vertex of defect $-1$. In particular, the corner whose corresponding involutions commute satisfies this property. This contradicts Lemma~\ref{lemma:commute3}. 
	
	If corner $c$ has defect $1$, the defect along $\gamma_w$ and $\gamma_v$ it $0$ as otherwise the  sum of the defects along the boundary of $S$ is less than $6$. By \cref{lem:counting}(\ref{item_4}) and \cref{lem:counting}(\ref{item_2}), $a_w$ and $b_v$ have defect $-1$ and the vertex closest to $c$ on $\gamma_v$ with nonzero defect has defect $1$. Recall that $c_w$ has defect $1$. By Lemma~\ref{lem:def1-corner_commuting}, the involutions corresponding to $c$ do not commute. Hence the involutions corresponding to $a$ or $b$ commute. But this contradicts Lemma~\ref{lemma:commute3}. 
	We have now shown that $a$ or $b$ has defect at most $1$. \qed\ claim 3.
	
	Recall that we suppose that the defect along all sides of $S$ is nonnegative and that $\gamma_u$ has defect one. We are now well-prepared to prove that both corners incident to $\gamma_u$ have defect at most $1$. 
	We assume for a contradiction that one of the corners $a$ and $b$, say $a$, has defect $2$. By claim $4$ we may assume that $b$ has defect at most $1$. We show the existence of a repeated swap surfaces of $S$ which do not satisfy all necessary properties. 
	

	By \cref{lem:counting}(\ref{item_7}) and \cref{lemma:shift}, there exists a surface $S'$ obtained from $S$ by an edge-swap along $\gamma_u$ in which $a_u$ has defect $1$.  By \cref{lem:swap_othersides}, the defect along $\gamma_u$ in $S'$ is $1$. By Lemma~\ref{lem:swap_corner}, $a$ has defect $2$ in $S'$ and the defect of $b$ is at most $1$ in $S'$. As above we conclude that $a_w$ has negative defect in $S'$ and that the defect along $\gamma_w$ is at most $0$.
	
	First we consider the case that $\gamma_w$ has negative defect in $S'$. Then the defect along $\gamma_u$  and $\gamma_v$ is $1$, $b$ has defect $1$ and $c$ has defect $2$ as otherwise the sum of the defects along the boundary of $S'$ is less than $6$.
	We apply \cref{lemma:shift} and conclude that there is a surface $\hat S$ that is obtained from $S'$ by an edge-swap along $\gamma_v$ such that $c_v$ has defect $1$ in $\hat S$. 
	By Lemma~\ref{lem:swap_corner} and Lemma~\ref{lem:inner def1-vertex}, the defects of the corners of $S'$ and the defect of $a_u$ do not change. By Lemma~\ref{lemma:defectes10} and Lemma~\ref{lemma:defects11}, vertices $a_w$ and $c_w$ have negative defect. If one of them would be smaller than $-1$, the defect along the side $\gamma_w$ would be at most $-2$ because of \cref{lem_sumdef} and the sum of the defects along the boundary of $\hat S$ would be less than $6$. Hence $a$ and $c$ are adjacent to a vertex of defect $1$ and a vertex of defect $-1$. Hence the corresponding involutions do not commute by Lemma~\ref{lemma:commute3}. But then the involutions corresponding to $b$ commute which contradicts \cref{lem:def1-corner_commuting}.
	
	It remains to consider the case where the side $\gamma_w$ in the surface $S'$ has defect $0$.  Then $a_w$ has defect $-1$  and the vertex closest to $c$  on $\gamma_w$ with nonzero defect has defect $1$ by \cref{lem:counting}(\ref{item_4}) and \cref{lem:counting}(\ref{item_2}). We apply \cref{lemma:shift} and obtain a surface $\hat S$ that is obtained from $S'$ by an edge-swap along $\gamma_w$ such that $c_w$ has defect $1$ in $\hat S$.   By Lemma~\ref{lem:swap_corner} and Lemma~\ref{lem:inner def1-vertex}, the defects of the corner $a$ and the defect of $a_u$ do not change. Hence, $\defect(a_u)=1$ and $\defect(a)=2$ in $\hat S$. By \cref{lem:swap_othersides}, $\defect(\gamma_w)=0$. As above follows that $\defect(a_w)=-1$ in $\hat S$. By \cref{lem:swap_corner}, the defect of $b$ is at most $1$ in $\hat S$.
	
	If the defect of $c$ in $\hat S$ is $1$, the defect along $\gamma_u$ and $\gamma_v$ in $\hat S$ is $1$ as otherwise the sum of the defects along the boundary of $S$ is less then $6$. Then the involutions corresponding to $b$ and $c$ do not commute by Lemma~\ref{lem:def1-corner_commuting}. Hence the involutions corresponding to $a$ commute which contradicts \cref{lemma:commute3}.
	
	If $c$ has defect $2$ in $\hat S$ , we argue like before that $c_v$ has defect $-1$.  Because $c_v$ has defect $-1$, $\defect(\gamma_v)\le 0$ because of \cref{lem_sumdef}. Recall that the defects along the boundary of $\hat S$ sum up to $6$. Hence  $\defect(\gamma_v)= 0$. By \cref{lem:counting}(\ref{item_4}), the vertex closest to $b$ on $\gamma_v$ with nonzero defect has defect $1$. As $\defect(a)=2$, $\defect(c)=2$, $\defect(b)\le1$, $\defect(\gamma_w)= \defect(\gamma_v)=0$ in $\hat S$, the defect along $\gamma_u$ is $1$ by \cref{prop:GaussBonnet}.  By \cref{lem:counting}(\ref{item_7}), the vertex on $\gamma_u$ closest to $b$ with nonzero defect has defect $1$. Then the involutions corresponding to $b$ do not commute by Lemma~\ref{lem:def1-corner_commuting}. 
	Note that the corners $a$ and $c$ are adjacent to a vertex of defect $1$ and to a vertex of defect $-1$. Hence the involutions corresponding to $a$ and $c$ do not commute by Lemma~\ref{lemma:commute3}. All in all neither $a$ nor $b$ nor $c$ is the corner of $\hat S$ whose corresponding involution commute which is a contradiction.  
\end{proof}

\begin{lemma}
	If every side of $S$ has at least one inner vertex, the defect along at least one side of $S$ is nonzero. 
	\label{lem:zero-sides}
\end{lemma}
\begin{proof}
	Assume for a contradiction that the defect along all sides is $0$. 
	If one corner would have a defect smaller than $2$, the sum of the defect along the boundary of $S$ would be at most $5$ which contradicts \cref{prop:GaussBonnet}. Thus every corner has defect $2$. We will show that at least one corner of $S$, say $a$, is adjacent to a vertex of nonzero defect. As the defect along every side is $0$, the neighbors of $a$ on the boundary of $S$ have defect at least $-1$ by \cref{lem:counting}(\ref{item_4}). By Lemma~\ref{lemma:defects11} it is impossible that both neighbors have defect $1$. By Lemma~\ref{lemma:defectes10} it is impossible that one neighbor has defect $1$ and the other has defect $0$. We conclude that at least one of the two neighbors of $a$ has defect $-1$. 
	
	Without loss of generality we assume that the neighbor $a_w$  of $a$ on $\gamma_w$ has defect $-1$. Then the vertex  on $\gamma_w$ which is closest to $c$ having nonzero defect, has defect $1$ by \cref{lem:counting}(\ref{item_2}). We apply \cref{lemma:shift} to $\gamma_w$ and obtain a surface $S'$ where the defect of $c_w$ is $1$. Then the neighbor of $c$ on $\gamma_v$ in $S'$ has negative defect by Lemma~\ref{lemma:defects11} and Lemma~\ref{lemma:defectes10}. As before we observe that its defect is $-1$. Furthermore the defects of $a$, $a_u$ and $a_w$ in $S'$ are the same like before. Any edge-swap does not change the defects of $a$ by Lemma~\ref{lem:swap_corner}. As $a_u$ has defect $1$ in $S$, it is not contained in one of the swapped edges and hence its defect does not change. The defect $a_w$ does not change by the definition of the edge-swap. We repeat the same argumentation for the side $\gamma_v$ and obtain a repeated swap-surface $\hat S$ of $S$ where every corner is adjacent  to a vertex of defect $1$ and a vertex of defect $-1$. In particular the corner whose corresponding involutions commute has this property. By Lemma~\ref{lemma:commute3} this corner has defect at most $1$ which is a contradiction. 
	
	\subfiglabelskip=0pt
	\begin{figure}
		\centering
		\subfigure[][]{\label{fig:vertex_00A}\includegraphics[]{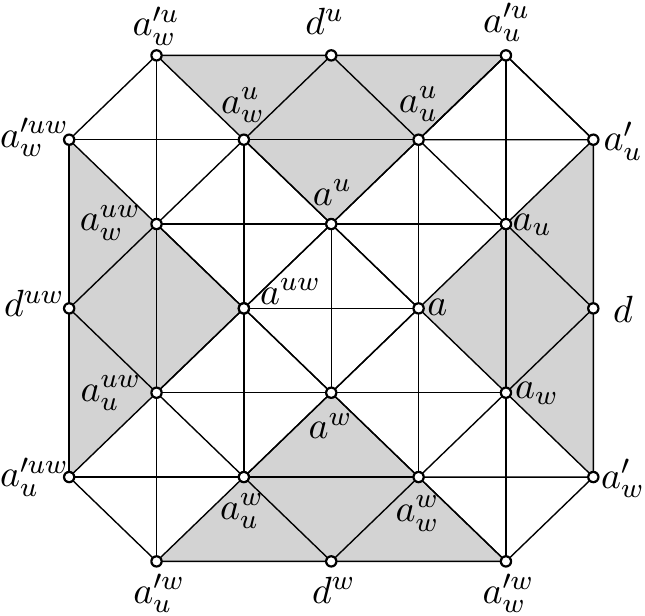}} 
		\subfigure[][]{\label{fig:vertex_00}\includegraphics[]{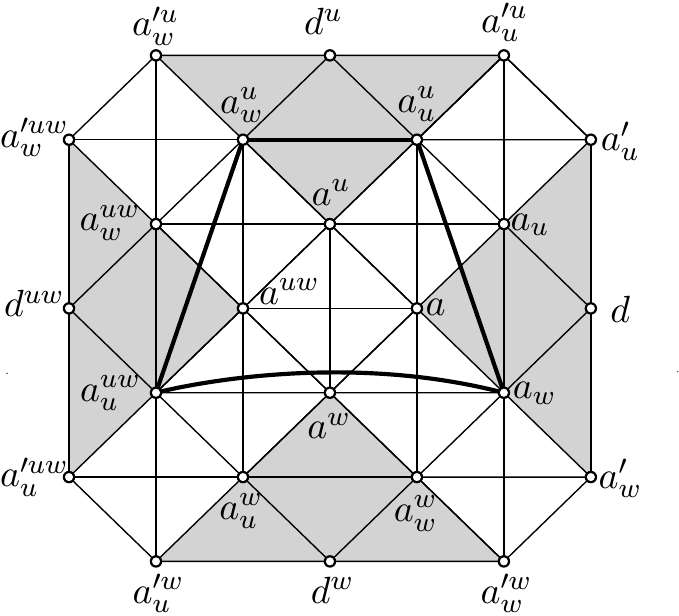}}	
		\caption{A surface $S$ studied in the proof of  Lemma~\ref{lem:zero-sides}. If $S$ has a corner $a$ that is incident to two vertices of defect $0$ then $X$ contains the subcomplex illustrated in \subref{fig:vertex_00A}.  The cycle in \subref{fig:vertex_00} is contained in $X$ and has no diagonal which contradicts $6$-largeness. }
		\label{ffig:vertex_00_commonfig}
	\end{figure}


	It remains to show that one corner of $S$ is adjacent to a vertex of nonzero defect. We show that the corner whose involutions commute has this property. Let $a$ be this corner. We assume for a contradiction  that $a_u$ and $a_w$ have defect $0$. Let $a_u'$ be the vertex on $\gamma_u$ adjacent  to $a_u$ different from $a$. Let $a_w$ be the vertex on $\gamma_w$ adjacent  to $a_w$ on $\gamma_r$ other than $a$. Let $d$ be the unique inner vertex of $S$ which is connected to $a_u$, $a_u'$, $a_w$ and $a_w'$. We obtain that $S$ contains the subcomplex shown in \cref{fig:vertex_00A}.
	
	By studying cycles of length $4$ and $5$ we will proof that $X$ contains the black thickened cycle pictured in Figure~\ref{fig:vertex_00} and that this cycle does not have diagonals which contradicts $6$-largeness of $X$.
	
	
	First we show that $X$ does not contain the edge $(a,d)$. To arrive a contradiction we assume that $X$ contains the edge $(a,d)$. Then the vertices $a$, $a_u$, $a_w$ and $d$ form a simplex in $X$. We exchange the edge $(a_u, a_v)$ with the edge $(a,d)$ in the $1$-skeleton of $S$ and obtain a new surface $S'$ with the same minimality properties like $S$ where $a$, $a_u$ and $a_w$ have defect $1$. This configuration contradicts Lemma~\ref{lemma:commute1}.
	
	We show that $X$ does not contain the edge $(a_u,a_w')$. To arrive a contradiction we assume that $X$ contains the edge $(a_u,a_w')$. Then the vertices $a_u$, $a_w$, $a_w'$ and $d$ form a simplex in $X$. We exchange the edge $(a_w,d)$ with the edge $(a_u,a_w')$ in the $1$-skeleton of $S$ and obtain a new surface $S'$ with the same minimality properties like $S$ where $a_w$ has defect $1$ and $a_u$ has defect $-1$. As the involutions corresponding to $a$ commute, this contradicts Lemma~\ref{lemma:commute3}.
	
	Similarly one can conclude that $X$ does not contain the edge $(a_w,a_u')$.
	
	But then the edges $(a,a_w')$ and  $(a,a_u')$ are not contained in $X$ as otherwise $(a,a_u,d,a_w')$ and $(a,a_w,d,a_u')$ would form $4$-cycles without diagonals. 
	
	Observe that $X$ contains the $5$-cycle $(a,a_u,d,a_w',a_w^w)$. As $X$ does not contain $(a,d)$, $(a_u,a_w')$ and $(a,a_w')$, $X$ contains the edges $(a_u,a_w^w)$ and $(d,a_w^w)$, as otherwise $X$ would contain a $4$- or $5$-cycle without diagonals. 
	
	The complex $X$ contains the $5$-cycle $(d^u, a_u'^u,a_u,a^u,a_w^u)$. We have that $a^u\sim a_w^u$, $a_w^u\sim d^u$ and $d^u\sim a_u'^u$ as the group action is simplicial and  $a_u'^u\sim a_u$ and $a_u\sim a^u$ by Lemma~\ref{lem:u-stableSplx}. $a_w^u \nsim a_u'^u$ as otherwise  $a_w\sim a_u'$. $d^u \nsim a^u$ as otherwise $d \sim a$. $a^u \nsim a_u'^u$ as otherwise $a \sim a_u'$. Hence $X$ contains the diagonals $(a_u,a_w^u)$ and $(a_u,d^u)$ as otherwise $X$ contains a $4$- or a $5$-cycle without diagonals.

	As $a_u \sim a_w^u$ and $a_u\sim a_w^w$ the complex $X$ contains the edges $(a_u^u, a_w)$ and $(a_u^w,a_w)$ and the $5$-cycle $(a_w^u, a_u^u, a_w, a_w^w, a_u^{uw})$. However, $X$  does not contain $(a_w^u,a_w)$ and the edge  $(a_u^u, a_u^{uw})$ as otherwise $a_w$ would be contained in $X_u$ and $S$ would not be minimal. Moreover, $X$ does not contain the edge $(a_w,a_u^{uw})$ as otherwise  $(a_w^u, a_u^u, a_w, a_u^{uw})$ would form a $4$-cycle without diagonals. Hence $X$ contains both  diagonals $(a_w^u, a_w^w)$ and $(a_w^w,a_u^u)$. 
	
	As $X$ contains the edges  $(a_u^u, a_w)$ , $(a_w^w,a_u^u)$ and  $(a_u^w,a_w)$ and because the action of $\Gamma$ is simplicial there exists the $4$-cycle $(a_w^u, a_u^u, a_w, a_u^{uw})$. Observe that $a_u^{uw} \nsim a_u^u $ and $a_w^u \nsim a_w$ as otherwise $a_u$ or $a_w$ would be contained in $X_u$ and $X_w$ and $S$ would not be minimal. Hence we have found a $4$-cycle without a diagonal which contradicts $6$-largeness. 
\end{proof}


\section{Proof of the fixedpoint theorem}\label{sec:mainProof}

This section contains the proof of \cref{thm:fixedPoint} which we restate below for ease of reference. 

\FixedpointThm*

\begin{proof} 
	Let in the following $S$ be a surface as constructed in Section~\ref{sec:minimalSurface}. Recall that $S$ was chosen minimally with respect to both perimeter and area.  
	We will prove the assertion by contradiction and hence assume that the action of $\Gamma$ {does not stabilize a simplex}. \cref{prop:existence} and \ref{prop:notasimplex} imply that $S$ is larger than a single $2$-simplex. We consider a series of cases for the defects of the sides of $S$.  In each of these cases we will either arrive at a situation that contradicts minimality or contradicting one of the numerous counting statements that hold for $S$ and the (sums of) defects of its vertices. 
	
	Systolicity of $S$ and the combinatorial Gauss-Bonnet Lemma \ref{prop:GaussBonnet} imply that the sum of the defects along the boundary curve $C$ of $S$ has to be at least $6$. Further recall that the defect along any side of $S$ was defined as the sum of the defects of all of its inner vertices, i.e. all vertices different from its endpoints. The two main ingredients are Propositions~\ref{prop: sum_of_defects>=0} and \ref{prop:defectes_along_1_sides}. One states that we can assume without loss of generality that the defect along any one of the sides equals $1$ or zero. The other says that two corners incident to a side have defect at most $1$, if the defect along the side is $1$. By definition, any corner has defect at most $2$.
	By \cref{prop: sum_of_defects>=0} we may assume that the defect along any one of the sides is either one or zero. 
	
	We are thus in one of the following remaining cases which are illustrated (from left to right) in Figure~\ref{fig:cases-mainProof}. 
	The defects of the sides of $S$ are
	\begin{itemize}
		\item[(a)]  $1$ on one side and $0$ on the two other sides.
		\item[(b)]  $1$ on two sides and $0$ on the third side.
		\item[(c)]  $1$ on each side.
		\item[(d)]  $0$ on each side and each side has inner vertices.
		\item[(e)]  $0$ on each side and one of the sides has no inner vertices.
	\end{itemize}
	
	\begin{figure}[h!]  
		\centering
		\includegraphics[width=0.9\textwidth]{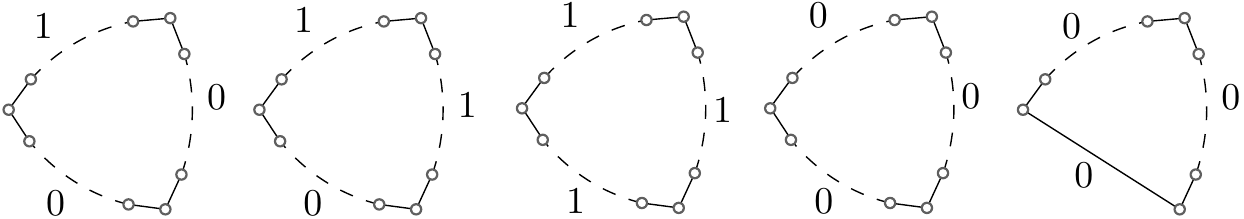}
		\caption{Cases (a) to (e) shown from left to right. }
		\label{fig:cases-mainProof}
	\end{figure}
	
	\cref{prop:defectes_along_1_sides} implies that in case (a) two corners have defect $1$ and in all other cases all corners have defect $1$. We deal with all the cases individually and will arrive at a contradiction in each of them. 
	
	\emph{Case (a):} 
	Two corners of $S$ have defect at most $1$. So the sum of the defects along the boundary of $S$ is at most $1+0+0+1+1+2 = 5$ which contradicts \cref{prop:GaussBonnet}. 
	
	\emph{Case (b):} 
	\cref{prop:defectes_along_1_sides} implies that any corner of $S$ has defect at most $1$. So the sum of the defects along the boundary of $S$ is at most $1+1+0+1+1+1 = 5$ which contradicts \cref{prop:GaussBonnet}.
	
	\emph{Case (c):} 
	Again by \cref{prop:defectes_along_1_sides} we obtain that the defects of corners is at most $1$. If only one corner has defect $0$, the sum of defects along the boundary of $S$ is at most $5$. Hence every corner has defect $1$. Because every vertex with nonzero defect  on the boundary of $S$ closest to any corner of $S$ has defect $1$  by Lemma~\itemref{lem:counting}{item_7}, we obtain by \cref{lem:def1-corner_commuting}  that there cannot be a coner for which the involutions commute which contradicts the shape of $\Gamma$. 
	
	\emph{Case (d):} 
	This case does not occur by \cref{lem:zero-sides}.
	
	\emph{Case (e):}
	If one corner has defect at most $1$, then the sum of the defects along the boundary of $S$ is at most $2+2+1=5$ which contradicts \cref{prop:GaussBonnet}. Therefore the defect of every corner is $2$. If no side of $S$ has inner vertices, $S$ is just a $2$-simplex which contradicts the result in \cref{sec:notasimplex} Hence, there is one side $\gamma$ of $S$ that contains an inner vertex. As one side of $S$ does not have inner vertices, one corner $a$ of $\gamma$ coincides with a corner of  a side $\gamma'$ of $S$ without inner vertices. 
	Let $c$ be the corner of $\gamma'$ other than $a$. 
	Since $a$ has defect $2$, corner $c$ is connected to an inner vertex of $\gamma$. 
	As this vertex is an inner vertex, $c$ is connected to at least one other vertex contained in $S$. But then $c$ has degree at least three and hence defect at most $1$ and we arrive at a contradiction.
	
	We have now proven that each case leads to a contradiction. Hence Theorem~\ref{thm:fixedPoint} follows.
\end{proof}

\cref{thm:nonsystolic} is now an easy consequence of \cref{thm:fixedPoint} as if the groups in question were systolic they would admit a geometric action on a systolic complex. However we have just seen that every (simplicial) action admits a fixed point and hence cannot be geometric.

\renewcommand{\refname}{Bibliography}
\bibliography{Bibliography}
\bibliographystyle{alpha}

\end{document}